\documentclass{amsart}

\author{Matthew D. Kvalheim}
\address{Department of Mathematics and Statistics, University of Maryland, Baltimore County, MD, USA}

\email{kvalheim@umbc.edu}

\title[Asymptotically stable vector fields and Lyapunov functions]
{Differential topology of the spaces of asymptotically stable vector fields and Lyapunov functions}

\usepackage{import} % Use import instead of input; for pdf_tex in figure folder

\usepackage[utf8]{inputenc}
\usepackage[T1]{fontenc}
\usepackage{lmodern}
\usepackage{amsmath}
\usepackage{amssymb}
\usepackage{amsthm}
\usepackage{stmaryrd} % for minnorm \llfloor, \rrfloor symbols
\usepackage{mathtools}
\usepackage{tikz-cd}
\usepackage{subfig}

\usepackage{thmtools}
\usepackage{thm-restate} % For restating theorem verbatim

\newcommand{\concept}[1]{\textbf{#1}}

\newcommand{\ct}{\setminus}
\newcommand{\R}{\mathbb{R}}
\newcommand{\sph}{S}
\newcommand{\N}{\mathbb{N}}

\newcommand{\coloneqq}{:=}
\newcommand{\id}{\textnormal{id}}

\newcommand{\vf}{\mathfrak{X}}

\newcommand{\rel}[1]{\textnormal{rel}\,#1}
\newcommand{\vl}{\textnormal{vl}}

\newcommand{\I}{\mathcal{I}}
\newcommand{\U}{\textnormal{Gr}}

\newcommand{\fun}{\mathcal{L}}
\newcommand{\stab}{\mathcal{S}}
\newcommand{\AN}{\mathcal{AN}}
\newcommand{\F}{\mathcal{F}}
\newcommand{\G}{\mathcal{G}}
\newcommand{\morsfun}{\mathcal{M}\fun}
\newcommand{\hypstab}{\mathcal{H}\stab}

\newcommand{\Diff}{\textnormal{Diff}}

\newcommand{\Emb}{\textnormal{Emb}}
\newcommand{\stabc}{\prescript{c}{}{\stab}}
\newcommand{\hypstabc}{\prescript{c}{}{\mathcal{H}\stab}}
\newcommand{\ANc}{\prescript{c}{}{\AN}}
\newcommand{\vfc}{\prescript{c}{}{\vf}}
\newcommand{\interior}[1]{\textnormal{int}(#1)}
\newcommand{\cl}[1]{\textnormal{cl}(#1)}

\newcommand{\GL}{GL}
\newcommand{\GS}{GS}
\newcommand{\PDS}{PDS}
\newcommand{\Hur}{\textnormal{Hur}}
\newcommand{\Fr}{\textnormal{Fr}}

\newcommand{\att}{A}

\newcommand{\vo}{F}
\newcommand{\vt}{G}

\newcommand{\bas}{B}
\newcommand{\lyap}{V}
\newcommand{\lyapt}{W}
\newcommand{\lyapth}{Y}

\newcommand{\param}{P}
\newcommand{\dom}{\textnormal{dom}}
\newcommand{\pr}{\textnormal{pr}}

\newcommand{\ip}[2]{\langle #1, #2 \rangle}
\DeclarePairedDelimiter\norm{\lVert}{\rVert}

\DeclareMathOperator\supp{supp}

\theoremstyle{definition}
\newtheorem{Lem}{Lemma}
\newtheorem*{Lem-non}{Lemma}
\newtheorem{Th}{Theorem}
\newtheorem{Co}{Corollary}
\newtheorem*{Co-non}{Corollary}

\newtheorem{Prop}{Proposition}
\newtheorem*{Prop-non}{Proposition}

\newtheorem*{Quest-non}{Question}

\newtheorem*{Obs-non}{Observation}

\newtheorem*{Th-non}{Theorem}% Theorem without numbering

\newcommand{\thistheoremname}{}
\newtheorem*{genericthm}{\thistheoremname}
{\renewcommand{\thistheoremname}{Theorem~\ref{#1}$'$}%
	\begin{genericthm}}
	{\end{genericthm}}

\newtheorem{Def}{Definition}
\newtheorem*{Def*}{Definition}

\newtheorem*{Ex-non}{Example}

\newtheorem{Rem}{Remark}

\usepackage{marvosym}
\usepackage[%hyperref
hypertexnames,%
citecolor=blue,%
colorlinks=true,%
linkcolor=red%
]{hyperref}
\usepackage[all]{hypcap}

\begin{document}
	%\sffamily
	
	\begin{abstract}
	We study the topology of the space of all smooth asymptotically stable vector fields on $\mathbb{R}^n$, as well as the space of all proper smooth Lyapunov functions for such vector fields.
	We prove that both spaces are path-connected and simply connected when $n\neq 4,5$ and  weakly contractible when $n\leq 3$.
	Moreover, both spaces have the weak homotopy type of the nonlinear Grassmannian of submanifolds of $\mathbb{R}^n$ diffeomorphic to the $n$-disc.
	
	The proofs rely on Lyapunov theory and  differential topology, such as the work of Smale and Perelman on the generalized Poincar\'{e} conjecture and results of Smale, Cerf, and Hatcher on the topology of diffeomorphism groups of discs.
    Applications include a partial answer to a question of Conley,  a parametric Hartman-Grobman theorem for nonhyperbolic but asymptotically stable equilibria, and a parametric Morse lemma for degenerate minima. 
    We also study the related topics of hyperbolic equilibria, Morse minima, and relative homotopy groups of the space of asymptotically stable vector fields inside the space of those vanishing at a single point.
	\end{abstract}
	
	\maketitle

\setcounter{tocdepth}{1}
	\tableofcontents

\section{Introduction}\label{sec:intro}

Lyapunov discovered that an equilibrium for a vector field is asymptotically stable if there exists what we now call a Lyapunov function \cite{lyapunov1892general}. %, and the equilibrium is globally asymptotically stable if the function is proper and has no other critical points \cite{barbasin1952stability}.
More than half a century later, Kurzweil and Massera proved converse theorems guaranteeing that $C^\infty$ Lyapunov functions always exist for asymptotically stable equilibria of reasonable vector fields  \cite{kurzweil1956inversion,massera1956contributions}.
Subsequently, Wilson studied the topology of level sets of such Lyapunov functions \cite{wilson1967structure}. 
In this paper, we turn this idea on its head by using (sub)level sets to study the topology of the \emph{space} of all Lyapunov functions, and in turn the topology of the space of all asymptotically stable vector fields.

To state our results precisely, let $\stab^r(\R^n)$ be the space of $C^r$  vector fields on $\R^n$ having a globally asymptotically stable equilibrium, where $1\leq r \leq \infty$.
Let $\fun^r(\R^n)$ be the space of surjective proper $C^r$ functions $\R^n\to [0,\infty)$ having a unique minimum and critical point.
Here both spaces are equipped with the compact-open $C^r$ topology.
One of our main results %(Corollary~\ref{co:varying-eq-min-trivial-homotopy-groups})
 is
\begin{Th}\label{th:intro-1-conn-contract}
	$\stab^r(\R^n)$ and $\fun^r(\R^n)$ are both path-connected and simply connected if $n\neq 4,5$ and weakly contractible if $n\leq 3$.
\end{Th}
This result concerns homotopies of continuous families of $C^r$ functions.
Or, by obstruction theory, it concerns the existence of continuous families $\param\to \stab^r(\R^n)$, $\param \to \fun^r(\R^n)$ restricting to given families on the boundary $\partial \param$ of a $C^\infty$ manifold $\param$.

Now, a continuous family of $C^r$ functions is not the same as a $C^r$ family.
However, we do obtain fully $C^r$ statements
%Given $C^\infty$ manifolds $A, B, C$ and a subset $\F \subset C^r(B,C)$, we say that a map $f\colon A \to \F$ is $C^r$ if the map $A\times B \to C$ defined by $(a,b)\mapsto f(a)(b)$ is. we say that a map $f\colon \param\to $
like the following, which can be interpreted as an existence theorem for a boundary value problem.
\begin{Th}\label{th:intro-cr-extensions}
Let $\param$ be a compact $C^\infty$ manifold with boundary.
Let $\vo_\partial \colon \partial \param \times \R^n\to \R^n$ and $\lyap_\partial\colon \partial \param \times \R^n\to [0,\infty)$ be $C^r$ maps satisfying $\vo_\partial(p,\cdot)\in \stab^r(\R^n)$ and $\lyap_\partial(p,\cdot)\in \fun^r(\R^n)$ for all $p\in \partial \param$.
Assume either (i) $n\neq 4,5$ and $\dim \param \leq 2$ or (ii) $n\leq 3$.
Then there are $C^r$ maps $\vo\colon \param \times \R^n\to \R^n$ and $\lyap\colon \param \times \R^n\to [0,\infty)$ that extend $\vo_\partial$ and $\lyap_\partial$ and satisfy $\vo(p,\cdot)\in \stab^r(\R^n)$ and $\lyap(p,\cdot)\in \fun^r(\R^n)$ for all $p\in \param$.
\end{Th}
Both theorems remain true for the subspace $\stabc^r(\R^n)\subset \stab^r(\R^n)$ of vector fields that are (backward) complete, as well as the subspaces $\stab_0^r(\R^n)$, $\stabc_0^r(\R^n)$, $\fun_0^r(\R^n)$ for which the equilibria or minima are located at the origin.  %(Corollaries~\ref{co:funnels-trivial-homotopy-groups}, \ref{co:vf-trivial-homotopy-groups}, \ref{co:complete-vf-trivial-homotopy-groups}, \ref{co:varying-eq-min-trivial-homotopy-groups}).
On the other hand, some restrictions on $n$ are necessary, and validity of the path-connectedness portion of Theorem~\ref{th:intro-1-conn-contract} for $n=5$ would imply that the $4$-dimensional smooth Poincar\'{e} conjecture is true.  %(section~\ref{sec:path-conn-poincare}).

Applications include a partial answer to a question of Conley, %(Theorem~\ref{th:local-path-conn}),  
a parametric Hartman-Grobman theorem for nonhyperbolic but asymptotically stable equilibria, %(Theorem~\ref{th:parametric-global-linearization}), 
and a parametric Morse lemma for degenerate minima.  %(Theorem~\ref{th:parametric-global-morseification}).
These three applications respectively extend results of Jongeneel \cite{jongeneel2024while},  Coleman \cite{coleman1965local,coleman1966addendum} (see also  \cite{kvalheim2025global}), and  Gr\"{u}ne, Sontag, and Wirth \cite{gruene1999twist}.

We also consider the space $\AN^r(\R^n)$ of $C^r$ vector fields on $\R^n$ vanishing at exactly one point and its subspace $\AN_0^r(\R^n)$ of vector fields vanishing at the origin.
We prove that the inclusion $\stab_0^r(\R^n)\hookrightarrow \AN_0^r(\R^n)$ is nullhomotopic.  %(Theorem~\ref{th:nullhomotopic-inclusion-s-an}).
This furnishes information about the relative homotopy groups $\pi_k(\AN^r_0(\R^n), \stab^r_0(\R^n))$, and we deduce similar information for the variants of vector fields that are complete and have unrestricted equilibrium location.
As an application, we illustrate how this information can be used to detect the presence of an unstable vector field within a parametric family of vector fields. %(\S \ref{subsec:obstructions-param-asymptotic-stab}). 

In the context of asymptotic stability, similar nullhomotopy arguments go back at least to Krasnosel'ski\u{\i} and Zabre\u{\i}ko  \cite{krasnoselskii1984geometrical}, and our nullhomotopy formula is adapted from one of Sontag \cite{sontag1998mathematical}.
In our language, these authors were concerned with proving that  $\stab_0^r(\R^n)$ is contained in a path component of $\AN_0^r(\R^n)$, and here we observe that the same argument works parametrically to prove more.
A similar argument in another context appeared in work of Eliashberg and Mishashev \cite{eliashberg2000wrinklingII}, who attribute the technique to Douady and Laudenbach \cite{laudenbach1976formes} and independently Igusa.

Finally, for completeness, in an appendix we consider the subspaces $\morsfun^r(\R^n)\subset \fun^r(\R^n)$ of Morse functions and $\hypstab^r(\R^n)\subset \stab^r(\R^n)$ of vector fields for which the equilibrium is hyperbolic.
In contrast to Theorem~\ref{th:intro-1-conn-contract}, these spaces are both weakly contractible for any $n$, as are their variants with complete vector fields and constrained locations of minima and equilibria.

\subsection{Outline of the proof of Theorems~\ref{th:intro-1-conn-contract} and \ref{th:intro-cr-extensions}}

We first prove Theorem~\ref{th:intro-1-conn-contract} for $\fun_0^\infty(\R^n)$.
To do so we observe that, for any $\lyap \in \fun^\infty(\R^n)$, the sublevel set $\lyap^{-1}([0,1])$ is diffeomorphic to the $n$-disc $D^n\subset \R^n$ if $n\neq 4,5$.
This follows from Smale's $h$-cobordism theorem for $n\geq 6$ \cite{smale1962structure}, Perelman's solution to the Poincar\'{e} conjecture for $n=3$ \cite{perelman2002entropy,perelman2003ricci,perelman2003finite}, and classical facts for $n\leq 2$.
We then consider the space $$\U(D^n,\R^n)= \Emb(D^n,\R^n)/\Diff(D^n)$$ of  submanifolds of $\R^n$ diffeomorphic to $D^n$,  known as a \emph{nonlinear Grassmannian} \cite{gaybalmaz2014principal}, and its open subset $\U_0(D^n,\R^n)$ of submanifolds whose interiors contain the origin.
Here $\U(D^n,\R^n)$ is topologized as the quotient of the space $\Emb(D^n,\R^n)$ of $C^\infty$ embeddings $D^n\hookrightarrow \R^n$ by the natural right action of the group $\Diff(D^n)$ of diffeomorphisms of $D^n$, where the latter spaces have the $C^\infty$ topology.

One of our main results %(Theorem~\ref{th:funnel-space-fiber-bundle-over-U})  
is
\begin{Th}\label{th:intro-grassmannian}
	For $n\neq 4,5$ the sublevel set map $$p\colon \fun_0^\infty(\R^n)\to \U_0(D^n,\R^n), \quad p(\lyap)\coloneqq\lyap^{-1}([0,1])$$ is a topological fiber bundle with weakly contractible fibers.
    In particular, $p$ is a weak homotopy equivalence.
\end{Th}

Next, since the inclusion 
\begin{equation*}
	\U_0(D^n,\R^n) \hookrightarrow \U(D^n,\R^n)
\end{equation*}
is a weak homotopy equivalence (Lemma~\ref{lem:gr-0-gr-w-h-e}),
 Theorem~\ref{th:intro-grassmannian} leads us to investigate the topology of $\U(D^n,\R^n)$.
Generalizing a theorem of Binz and Fischer \cite{binz1981manifold} based on an idea implicit in work of Weinstein \cite{weinstein1971symplectic}, Gay-Balmaz and Vizman \cite{gaybalmaz2014principal} proved that the natural quotient map
$$\Emb(D^n,\R^n)\to \U(D^n,\R^n), \quad f\mapsto f(D^n)$$
is a principal $\Diff(D^n)$-bundle, hence also a Serre fibration.
Using the long exact sequence of homotopy groups of this fibration together with results of Smale \cite{smale1959diffeomorphisms}, Cerf \cite{cerf1970stratification}, and Hatcher \cite{hatcher1983proof} related to the homotopy groups of $\Diff(D^n)$, we  deduce that $\U(D^n,\R^n)$ is path-connected and simply connected if $n\neq 4,5$ and weakly contractible if $n\leq 3$.
This and Theorem~\ref{th:intro-grassmannian} establish Theorem~\ref{th:intro-1-conn-contract} for $\fun_0^\infty(\R^n)$.

To finish the proof of Theorem~\ref{th:intro-1-conn-contract} for the other spaces, we prove that the inclusions
\begin{equation*}
	\begin{tikzcd}
		\fun_0^\infty(\R^n) \arrow[hookrightarrow]{r} & \fun_0^r(\R^n) \arrow[hookrightarrow]{r} & \fun^r(\R^n)
	\end{tikzcd}
\end{equation*}
are weak homotopy equivalences for $1\leq r\leq \infty$, as are all arrows in the diagram
\begin{equation*}
	\begin{tikzcd}
		\fun_0^{r+1}(\R^n) \arrow[hookrightarrow, "-\nabla"]{r} \arrow[hookrightarrow]{d}& \stab_0^r(\R^n) \arrow[hookleftarrow]{r} \arrow[hookrightarrow]{d} & \stabc_0^r(\R^n) \arrow[hookrightarrow]{d}\\
		\fun^{r+1}(\R^n)\arrow[hookrightarrow,, "-\nabla"]{r} & \stab^r(\R^n) \arrow[hookleftarrow]{r} & \stabc^r(\R^n)
	\end{tikzcd}.
\end{equation*}
Here the arrows are embeddings given by ``take the negative gradient''  and subspace inclusion.
This and Theorem~\ref{th:intro-grassmannian} imply
\begin{Th}\label{th:intro-weak-htpy-vfs-etc}
If $n\neq 4,5$, then $\fun_0^r(\R^n)$, $\fun^r(\R^n)$, $\stab_0^r(\R^n)$, $\stab^r(\R^n)$, $\stabc_0^r(\R^n)$, and $\stabc^r(\R^n)$ all have the weak homotopy type of $\U(D^n,\R^n)$.
\end{Th}
Thus, all homotopy groups of these spaces are isomorphic to those of $\U(D^n,\R^n)$.
This and the aforementioned result on the topology of $\U(D^n,\R^n)$ finishes the proof of Theorem~\ref{th:intro-1-conn-contract} for all cases.

To prove that the maps in the diagrams above are weak homotopy equivalences, we rely on smooth approximations \cite{hirsch1976differential},  Wilson's converse Lyapunov theorem for asymptotically stable compact invariant sets  \cite{wilson1969smooth}, and a smoothing technique of Fathi and Pageault \cite{fathi2019smoothing}.
The same techniques are used to prove Theorem~\ref{th:intro-cr-extensions}.

\subsection{Organization}

The rest of the paper is organized as follows. 
In section \ref{sec:prelim} we explain our conventions, definitions, and notation.
Section~\ref{sec:fiber-bundle-weak-htpy-type} is the  heart of the paper, containing the proof of Theorem~\ref{th:intro-grassmannian} and subsequently Theorems~\ref{th:intro-1-conn-contract}, ~\ref{th:intro-cr-extensions} for $\fun_0^r(\R^n)$.
Section~\ref{sec:htpy-groups-gas-vf} uses Lyapunov theory to complete the proof of Theorems~\ref{th:intro-1-conn-contract}, ~\ref{th:intro-cr-extensions} for $\stab_0^r(\R^n)$. 
(These and other sections also contain more refined results.)
In section~\ref{sec:path-conn-poincare} we explain why validity of the path-connectedness portion of Theorem~\ref{th:intro-1-conn-contract} for $n=5$ would imply that the $4$-dimensional smooth Poincar\'{e} conjecture is true.
Section~\ref{sec:relative-htpy-groups} contains the mentioned nullhomotopy result and consequences for relative homotopy groups.
In section~\ref{sec:varying-min-eq} we extend the previous results by removing the constraint on the location of the minima or equilibria.
Section~\ref{sec:applications} contains the mentioned applications.
In the appendix we consider the cases of Morse minima and hyperbolic equilibria.

\subsection*{Acknowledgments}
This material is based upon work supported by the Air Force Office of Scientific Research under award number FA9550-24-1-0299.
The author would like to thank Ryan Budney, Wouter Jongeneel, Alexander Kupers, Mike Miller Eismeier,  and \'Alvaro del Pino G\'omez for helpful conversations. 

\section{Preliminaries}\label{sec:prelim}

\subsection{Subsets of $\R^n$ and $\R^{n\times n}$}
We let $I$ denote the unit interval $[0,1]\subset \R$ and, when $n$ is clear from context, $B_r(y)$ denote the open ball $\{x\in \R^n\colon \norm{x-y}< r\}$ of radius $r>0$ centered at $y\in \R^n$.
We denote by $D^n=\{x\in \R^n\colon \norm{x}\leq 1\}$ the standard unit $n$-disc and $\sph^{n-1}=\partial D^n = \{x\in \R^n\colon \norm{x}=1\}$ the standard $(n-1)$-sphere.
Throughout this paper we assume $n>0$.

We denote by $\GL(n)\subset \R^{n\times n}$ the general linear group of invertible $n\times n$ real matrices,  $O(n)\subset \GL(n)$ the group of orthogonal matrices, and $\GL^+(n)\subset \GL(n)$ and $SO(n)\subset O(n)$ those matrices with positive determinant.

%We denote by $\PDS(n)\subset \GL^+(n)$ the set of positive definite symmetric matrices and $\Hur(n)\subset \GL(n)$ the set of real Hurwitz matrices (those whose possibly complex eigenvalues all have negative real part).

\subsection{Function spaces and nonlinear Grassmannians}\label{subsec:func-spaces-non-grass}

\subsubsection{Generalities}\label{subsubsection:generalities}
Given (finite-dimensional) $C^\infty$ manifolds with (possibly empty) corners $M$, $N$ and $r\in \N_{\geq 0}\cup \{\infty\}$, we let $C^r(M,N)$ denote the set of $C^r$ maps $M\to N$.
Let $\G$ be any subset of $C^r(M,N)$.
Unless stated otherwise, we equip any such $\G$ with the subspace topology induced by the compact-open $C^r$ topology on $C^r(M,N)$.
This is the topology of uniform convergence  on compact sets of not only the functions but also all derivatives up to order $r$ \cite[sec.~2.1]{hirsch1976differential}.

Thus, if $\param$ is a topological space, a map $g\colon \param \to \G$ is continuous if and only if the $x$-partial $r$-jet map 
\begin{equation*}
\param \times M\to J^r(M,N), \quad (p,x)\mapsto J^r_{g(p)}(x)	
\end{equation*}
is continuous (\cite{fox1945topologies}, \cite[p.~62]{hirsch1976differential}).
More concretely, $g\colon \param \to \G$ is continuous if and only if for any $q\in \param$, $y\in M$, and local coordinates centered at $y\in M$ and $g(q)(y)\in N$, the map $(p,x)\mapsto g(p)(x)$ and its first $r$ partial derivatives with respect to $x$ are continuous in $(p,x)$.

We next define several specific $\G$.

\subsubsection{Lyapunov-like functions}
Given $r\in \N_{\geq 1}\cup\{\infty\}$,  let $\fun^r(\R^n)\subset C^r(\R^n,[0,\infty))$ be the space of all surjective proper $C^r$ functions $\R^n\to [0,\infty)$ having a unique minimum and critical point with value $0$.
Given $y\in \R^n$, we let $\fun_y^r(\R^n)\subset \fun^r(\R^n)$ denote those functions having their unique minimum at $y$.
Clearly  $\fun_y^r(\R^n)$ is the same as $\fun_0^r(\R^n)$ up to translation, so we mainly consider the latter space for simplicity.

\subsubsection{Vector fields and flows}

A locally Lipschitz vector field $\vo$ on a $C^\infty$ manifold $M$ generates a unique maximal local flow $\Phi\colon \dom(\Phi)\subset \R \times M\to M$ that is locally Lipschitz and $C^r$ if $\vo$ is $C^r$ for $r\in \N_{\geq 1}\cup \{\infty\}$.
The \concept{trajectory} of $\vo$ through $x\in M$ is the maximally defined curve $t\mapsto \Phi^t(x)$. 
The vector field is \concept{forward complete} if $\dom(\Phi)\supset [0,\infty)\times \R^n$, \concept{backward complete} if $\dom(\Phi)\supset (-\infty,0] \times \R^n$, and \concept{complete} if it is forward and backward complete.

A subset $\att\subset M$ is \concept{invariant} for $\vo$ (or $\Phi$) if all trajectories through points in $\att$ are contained in $\att$.
A compact invariant set $\att\subset M$ is \concept{(locally) asymptotically stable} for $\vo$ (or $\Phi$) if for every open set $U_1\supset \att$ there is a smaller open set $U_2\supset \att$ such that $U_2\times [0,\infty) \subset \dom(\Phi)$ and $\Phi^t(z)\to \att$ as $t\to\infty$ for all $z\in U_2$.
The \concept{basin of attraction} $\bas\subset M$ is the open set defined to be the set of all points $x\in M$ such that $[0,\infty)\times \{x\}\in \dom(\Phi)$ and $\Phi^t(x)\to \att$ as $t\to\infty$.
We say that an asymptotically stable compact invariant set $\att$ is \concept{globally asymptotically stable} if $\bas = M$ (in particular, the vector field is forward complete in this case).
This is possible only when the topologies of $\att$ and $M$ are suitably compatible, such as the case that $M=\R^n$ and $\att$ is a single point.
In the latter case $\att$ is an \concept{equilibrium}, that is, a point $y\in M$ such that $\vo(y)=0$.

Given $r\in \N_{\geq 0}\cup \{\infty\}$, let $\vf^r(M)$ be the space of $C^r$ vector fields on $M$.
When $M=\R^n$, there is no essential difference between $\vf^r(\R^n)$ and $C^r(\R^n,\R^n)$

We say that a vector field is \concept{almost nonsingular} if it has exactly one equilibrium.
Let $\AN^r(\R^n)\subset \vf^r(\R^n)$ be the space of $C^r$ almost nonsingular vector fields  on $\R^n$, and $\AN_y^r(\R^n)\subset \AN^r(\R^n)$ be those with equilibrium at $y\in \R^n$.

Given $r\in \N_{\geq 1}\cup \{\infty\}$, we let $\stab^r(\R^n)\subset \AN^r(\R^n)$ denote the space of $C^r$ globally asymptotically stable vector fields on $\R^n$, and $\stabc^r(\R^n)\subset \stab^r(\R^n)$ denote those that are additionally backward complete (recall that forward completeness is automatic).
Given $y\in \R^n$, Let $\stab_y^r(\R^n)\subset \stab^r(\R^n)$,  $\stabc_y^r(\R^n)\subset \stabc^r(\R^n)$ be those vector fields having their unique equilibrium at $y$.
Clearly  $\stab_y^r(\R^n)$, $\stabc_y^r(\R^n)$ are the same as $\stab_0^r(\R^n)$, $\stabc_0^r(\R^n)$ up to translation, so we focus on the latter for simplicity.

%\MK{Should I also define the Morse and hyperbolic spaces from that section here? I think I should if I define $\PDS(n)$ and $\Hur(n)$ here as I did above. But maybe I should not define any of these here bc the morse/hyperbolic section is independent from rest of paper, so this could be an extra distraction at this stage.}

\subsubsection{Diffeomorphisms and embeddings}\label{subsubsection:diffeos-embeddings}
Given a $C^\infty$ manifold with (possibly empty) boundary $M$ and a $C^\infty$ manifold $N$, we let $\Diff(M)\subset C^\infty(M,M)$ be the space of $C^\infty$ diffeomorphisms of $M$ and $\Emb(M,N)\subset C^\infty(M,N)$ be the space of $C^\infty$ embeddings of $M$ into $N$.
If $M$ and $N$ are oriented and $\dim M = \dim N$, then $\Diff^+(M)\subset \Diff(M)$ and $\Emb^+(M,N)$ denote the spaces of orientation-preserving diffeomorphisms and embeddings, respectively.
When $M=D^n$ and $N=\R^n$, we let $\Emb_0(D^n,\R^n)$ denote those embeddings $f$ mapping $\interior{D^n}$ onto a set containing the origin, that is, $0\in f(\interior{D^n})$.
We define $$\Emb_{0}^{+}(D^n,\R^n)\coloneqq \Emb_{0}(D^n,\R^n)\cap \Emb^{+}(D^n,\R^n).$$

\subsubsection{Nonlinear Grassmannians}\label{subsubsection:nonlinear-grass}
We follow the presentation in \cite[p.~296]{gaybalmaz2014principal}.
Let $M$ be a $C^\infty$ manifold with boundary and $N$ be a $C^\infty$ manifold.
The \concept{nonlinear Grassmannian} of submanifolds of $N$ of type $M$ is the set $\U(M,N)$ defined by 
\begin{equation*}
	\U(M,N)=\{S\subset N\colon S \text{ is a $C^\infty$ embedded submanifold of } N \textnormal{ diffeomorphic to } M\}.
\end{equation*}
Given $S\in \U(M,N)$, by definition there exists $f\in \Emb(M,N)$ such that $S=f(M)$.
The embedding $f$ is unique up to composition on the right by a diffeomorphism in $\Diff(M)$.
Thus, there is a bijection
\begin{equation*}
	\Emb(M,N)/\Diff(M) \to \U(M,N), \quad [f]\mapsto f(M).
\end{equation*}
Using this identification, we equip $\U(M,N)$ with the subspace topology induced by the quotient topology on $\Emb(M,N)/\Diff(M)$.

When $M$ is compact, it is known that $\Diff(M)$, $\Emb(M,N)$, and $\U(M,N)$ are $C^\infty$ Fr\'{e}chet manifolds and that the quotient map $\Emb(M,N)\to \U(M,N)$ is a $C^\infty$ Fr\'{e}chet principal $\Diff(M)$-bundle \cite[Thm~2.2]{gaybalmaz2014principal} (this was proved earlier for other classes of $M$; see \cite[sec.~44]{kriegl1997convenient} and references therein).
However, we shall use only that the quotient map is a topological fiber bundle with fiber $\Diff(M)$, hence a Serre fibration in particular.

We define $\U_0(D^n,\R^n)\subset \U(D^n,\R^n)$ to be the open subset of submanifolds of $\R^n$ whose interiors contain the origin.
Since $\Emb_0^+(D^n,\R^n)$ is the preimage of $\U_0(D^n,\R^n)$ under the quotient map, it follows that the restriction of the quotient map to $\Emb_0^+(D^n,\R^n)$ is a topological fiber bundle over $\U_0(D^n,\R^n)$ with fiber $\Diff^+(D^n)$.

\subsection{Maps into function spaces}\label{subsection:maps-into-function-spaces}
Given $C^\infty$ manifolds with corners $M$, $N$ and $r\in \N_{\geq 1}\cup \{\infty\}$, let $\G$ be any subset of $C^r(M,N)$.
Recall that $\G$ is equipped with the topology inherited from the compact-open $C^r$ topology on $C^r(M,N)$ unless stated otherwise.
The following definition is expedient for our purposes.

If $s\in \N_{\geq 0} \cup \{\infty\}$ satisfies $s\leq r$ and  $\param$ is a $C^\infty$ manifold with corners, we say that $g\colon \param\to \G$ is $C^s$ (or $g\in C^s$) if  the induced map $\param \times M \to N$  given by $(p,x)\mapsto g(p)(x)$ is a $C^s$ map in the usual sense (is continuous and has continuous partial derivatives up to order $s$ in local coordinates).

\textbf{Warning:}
\begin{itemize}
	\item If $g\colon \param\to \G$ is $C^r$, then $g$ is automatically continuous (with respect to the compact-open $C^r$ topology on $\G$), but this is not  the case if $g$ is $C^s$ with $s < r$.
	Thus, we need to specify that a map $g\colon \param \to \G$ is both continuous and $C^s$ if we want to indicate that it is continuous (with respect to the compact-open $C^r$ topology on $\G$) and also $C^s$ in the above sense. 
	\item In particular, when we say that a map $g\colon \param \to \G$ is $C^0$, we mean that the map $\param \times \R^n \to \R^n$ given by $(p,x)\mapsto g(p)(x)$ is $C^0$ (continuous), which is weaker than saying that $g\colon \param \to \G$ is continuous.
	\item On the other hand, a continuous map $g\colon \param \to \G$ is automatically $C^0$, but we will sometimes use redundant conditions like ``continuous and $C^0$'' to streamline the statements of some results (e.g. Lemma~\ref{lem:smoothing-funnels}).
\end{itemize}

\section{Topology of spaces of Lyapunov-like functions}\label{sec:fiber-bundle-weak-htpy-type}

This section contains the proof of  Theorems~\ref{th:intro-1-conn-contract}, \ref{th:intro-cr-extensions}, \ref{th:intro-grassmannian} for $\fun_0^r(\R^n)$.
Section~\ref{subsec:alexander} contains a lemma resembling a two-sided version of Alexander's trick \cite{alexander1923deformation}, but for  $C^\infty$ real-valued functions rather than homeomorphisms. 
Section~\ref{subsec:smoothing-funnels} contains a smoothing lemma implying that the inclusions $\fun_0^r(\R^n)\hookrightarrow \fun_0^s(\R^n)$, $\fun^r(\R^n)\hookrightarrow \fun^s(\R^n)$ are weak homotopy equivalences for $r>s$.
Section~\ref{subsec:contractibility-path-conn} contains further preparatory lemmas needed for the first main result in section~\ref{subsec:fiber-bundle-structure-weak-homotopy-type}. 
In section~\ref{subsec:fiber-bundle-structure-weak-homotopy-type} we show that for $n\neq 4,5$ the sublevel set map $\fun_0^\infty(\R^n)\to \U_0(D^n,\R^n)$ is well-defined and in fact a fiber bundle with weakly contractible fibers.
The proof of weak contractibility relies on a lemma of section~\ref{subsec:contractibility-path-conn}, which in turn relies on lemmas of sections~\ref{subsec:alexander} and \ref{subsec:smoothing-funnels}.
In section~\ref{subsec:vanishing-htpy-groups-funnels} we study the topology of $\U_0(D^n,\R^n)$ to complete the proof of Theorem~\ref{th:intro-1-conn-contract} and, using the smoothing lemma of section~\ref{subsec:smoothing-funnels}, the proof of Theorem~\ref{th:intro-cr-extensions}.

\subsection{Alexander-type trick}\label{subsec:alexander}

\begin{Lem}\label{lem:funnel-alexander}
	Let $\param$ be a $C^\infty$ manifold and $\lyap, \lyapt, J\colon \param \to \fun_0^\infty(\R^n)$ be $C^\infty$ maps such that $J(p)(x)=\lyap(p)(x)$ for all $(p,x)\in \{\lyap(p)(x)\leq a\}$ and $J(p)(x)=\lyapt(p)(x)$ for all $(p,x)\in \{b\leq \lyapt(p)(x)\}$ for some $b>a>0$.
	Then there is a $C^\infty$ homotopy from $\lyap$ to $\lyapt$.
	In fact,  there are $C^\infty$ homotopies $\alpha \colon I\times \param \to \fun_0^\infty(\R^n)$ from $\lyap$ to $J$ and $\beta \colon I\times \param \to \fun_0^\infty(\R^n)$ from $\lyapt$ to $J$ such that $\alpha_t(p)(x)=\lyap(p)(x)$ for all $(t,p,x)\in I\times \{\lyap(p)(x)\leq a\}$ and $\beta_t(p)(x)=\lyapt(p)(x)$ for all $(t,p,x)\in I \times \{b\leq \lyapt(p)(x)\}$. 
\end{Lem}

%\iffalse
\begin{proof}
	Let $\Phi(p)$ and $\Psi(p)$  be the $C^\infty$ maximal local flows of $\nabla \lyap(p)/\norm{\nabla \lyap(p)}^2$ and $\nabla \lyapt(p)/\norm{\nabla \lyapt(p)}^2$, respectively, on $\R^n \setminus \{0\}$.
	Then
	\begin{equation}\label{eq:comp-shift}
		\lyap(p) \circ \Phi(p)^t = \lyap(p) + t \quad \textnormal{and} \quad \lyapt(p) \circ \Psi(p)^t = \lyapt(p) + t,
	\end{equation}
	so their domains are 
	\begin{equation}\label{eq:Phi-Psi-domains}
		\begin{split}
			\dom(\Phi(p))&=\{(t,x)\colon x\neq 0, \lyap(p)(x) > - t\} \qquad \text{and}\\
			\dom(\Psi(p))&=\{(t,x)\colon x\neq 0, \lyapt(p)(x) > -t\}.
		\end{split}
	\end{equation}
	
 By \eqref{eq:Phi-Psi-domains} we may define $\alpha, \beta^0\colon I\times \param \to C^\infty(\R^n, [0,\infty))$ by
	\begin{equation*}
		\alpha_t(p)(x)= 
		\begin{cases}
			\frac{1}{t}J(p)\circ \Phi(p)^{t\lyap(p)(x)-\lyap(p)(x)}(x),  & 0< t \leq 1 \textnormal{ and } x\neq 0\\
			\lyap(p)(x), & t = 0	\textnormal{ or } x = 0
		\end{cases}
	\end{equation*}
	and
	\begin{equation*}
		\beta^0_t(p)(x)= 
		\begin{cases}
			tJ(p)\circ \Psi(p)^{\frac{1}{t}\lyapt(p)(x)-\lyapt(p)(x)}(x),  & 0< t \leq 1 \textnormal{ and } x\neq 0\\
			\lyapt(p)(x), & t = 0	\textnormal{ or } x=0
		\end{cases}.
	\end{equation*}	
	
   Note that $\alpha_t(p)(x)=\lyap(p)(x)$ for all $(t,p,x)$ in $I\times \{\lyap(p)(x)\leq a\}$ and also all $(t,p,x)$ in a neighborhood of $\{t=0\}$ by \eqref{eq:comp-shift}, so $\alpha$ is $C^\infty$.
   And $\nabla(\alpha_t(p))(x), \nabla(\beta^0_t(p)(x))\neq 0$ for all $x\neq 0$, since for each fixed $t\in (0,1]$ the  level sets of $\lyap$ and $\lyapt$ are diffeomorphically permuted by
   \begin{equation}\label{eq:alpha-beta-phi-psi}
   	(p,x)\mapsto \Phi^{t\lyap(p,x)-\lyap(p,x)}(x) \quad \textnormal{and} \quad (p,x)\mapsto \Psi^{\frac{1}{t}\lyapt(p,x)-\lyapt(p,x)}(x)
   \end{equation}
   at nonzero speed, so the maps \eqref{eq:alpha-beta-phi-psi} define diffeomorphisms of $P\times (\R^n\setminus \{0\})$.
   Thus, $\alpha \colon I\times \param \to \fun_0^\infty(\R^n)$ is the desired $C^\infty$ homotopy from $\alpha_0 = \lyap$ to $\alpha_1=J$.
   
  By \eqref{eq:comp-shift}, $\beta^0_t(p)(x)=\lyapt(p)(x)$ for (i) all $(t,p,x) \in  I \times \{b\leq \lyapt(p)(x)\}$  and also (ii) all $(t,p,x)$ in a neighborhood of $\{0\}\times \param \times (\R^n\setminus \{0\})$.  
   By (ii) the map $(t,p,x)\mapsto \beta^0_t(p)(x)$ is $C^\infty$ on $I\times \param \times (\R^n\setminus \{0\})$, and by \eqref{eq:comp-shift} the same map is continuous on all of $I\times \param \times \R^n$.
   
Let $\psi\colon I\to I$ be $C^\infty$, $0$ near $0$, and $1$ near $1$.
Then $\beta^1\colon I\times \param \to C^0(\R^n,[0,\infty))$ defined by $\beta^1_t(p)(x)=\beta^0_{\psi(t)}(p)(x)$ has the same properties as $\beta^0$, but $\beta^1_t$ coincides with $\lyapt$ for $t$ in an open neighborhood $N_0\subset I$ of $0$ and with $J$ for $t$ in an open neighborhood $N_1\subset I$ of $1$.
Let $\varphi\colon [0,\infty)\to [0,\infty)$ be a $C^\infty$ increasing homeomorphism such that (i) $\varphi^{-1}(0)=\{0\}$, (ii) $\varphi'(s)>0$ for $s > 0$, (iii) $\varphi(s)=s$ for $s\geq b$, and (iv) $\beta^2_t(p)(x)\coloneqq \varphi(\beta^1_t(p)(x))$ is $C^\infty$ and hence also $\fun_0^\infty(\R^n)$-valued \cite[Lem.~6.3]{fathi2019smoothing}.
Note that (iii) and the corresponding properties of $\beta^0, \beta^1$ imply that $\beta^2_t(p)(x)=\lyapt(p)(x)$ for all $(t,p,x) \in  I \times \{b\leq \lyapt(p)(x)\}$.

  Let $\rho\colon I\to I$ be a $C^\infty$ function equal to $1$ on a neighborhood of $I\setminus (N_0\cup N_1)$ and $0$ on a neighborhood of $\partial I$.
  Then $\beta\colon I\times \param \to C^\infty(\R^n,[0,\infty))$ defined by   \begin{equation}\label{eq:beta-def}
   	\beta_t = \rho(t)\beta^2_t+(1-\rho(t))\beta^1_t
   \end{equation}
   is a $C^\infty$ homotopy from $\lyapt$ to $J$.
   Moreover, $\beta$ is $\fun_0^\infty(\R^n)$-valued since
   \begin{align*}
   	\nabla(\beta_t(p))(x)&=\rho(t)\nabla(\beta^2_t(p))(x)+(1-\rho(t))\nabla(\beta^1_t(p))(x)\\
   	&= \left(\rho(t)\varphi'(\beta^1_t(p)(x))+1-\rho(t)\right)\nabla(\beta^1_t(p))(x)
   \end{align*}
   is well-defined and nonzero for $x\neq 0$.
   And $\beta_t(p)(x)=\lyapt(p)(x)$ for all $(t,p,x) \in  I \times \{b\leq \lyapt(p)(x)\}$ since each $\beta_t$ is a convex combination of maps with the same property. 
   Thus, $\beta\colon I\times \param \to \fun_0^\infty(\R^n)$ is the desired $C^\infty$ homotopy from $\beta_0=\lyapt$ to $\beta_1=J$. 

   A $C^\infty$ homotopy $\gamma\colon I\times \param \to \fun_0^\infty(\R^n)$ from $\gamma_0 = \lyap$ to $\gamma_1 = \lyapt$ is then given by smoothly concatenating $\alpha$ followed by $\beta$, namely, 
   \begin{equation}\label{eq:gamma-def}
   	\gamma_t = 
   	\begin{cases}
   		\alpha_{\psi(2t)} & 0\leq t \leq \frac{1}{2}\\
   		\beta_{\psi(2-2t)} & \frac{1}{2} \leq t \leq 1
   	\end{cases}.
   \end{equation}
\end{proof}
%\fi

\subsection{Smoothing lemma}\label{subsec:smoothing-funnels}

Recall that a $C^1$ embedded submanifold $S$ of a manifold with boundary $\param$ is \concept{neat} if $\partial S=S\cap \partial \param$ and $T_p S \not \subset T_p(\partial \param)$ for all $p\in \partial S$ \cite[pp.~30--31]{hirsch1976differential}. 

\begin{Lem}\label{lem:smoothing-funnels}
Fix $r,u\in \N_{\geq 1}\cup\{\infty\}$ and $s,v\in \N_{\geq 0}\cup \{\infty\}$ satisfying $r\geq s \geq v$ and $r\geq u \geq v$.
Let $\param$ be a compact $C^\infty$ manifold with boundary, $S\subset \param$ be a closed subset and $C^\infty$ neatly embedded submanifold of either $\param$ or $\partial \param$, and $Q\subset \param$ be an open neighborhood of $S$.
Let $\lyap\colon \param\to \fun^u(\R^n)$ be a continuous and $C^v$ map such that
\begin{itemize}
	\item $\lyap$ restricts to a $C^s$ map $S \to \fun^r(\R^n)$ on $S$ and
	\item $(p,x)\mapsto \lyap(p)(x)\in \R^n$ restricts to a $C^s$ map with $r$ continuous partial $x$-derivatives on a neighborhood of a closed subset $R\subset \param \times \R^n$ disjoint from both $Q\times \R^n$ and $Z \coloneqq \{(p,x)\in \param \times \R^n \colon \lyap(p)(x)=0\}$.
\end{itemize}
Then for any $\varepsilon > 0$, $u'<u+1$, and compact set  $K \subset (\param \times \R^n) \setminus Z$, there is a continuous and $C^v$ map $H\colon I\times \param \to \fun^u(\R^n)$ that is $\fun^u_0(\R^n)$-valued if $\lyap$ is and satisfies
\begin{itemize}
	\item $\max_{0\leq i \leq u'} \max_{(p,x)\in K}\norm{\partial_x^iH_t(p)(x)-\partial_x^i \lyap(p)(x)} < \varepsilon$ for all $t\in I$,
    \item $H$ is a homotopy $\rel{S}$ from $H_0 = \lyap$ to a continuous and $C^s$ map $H_1\colon \param \to \fun^r(\R^n)$, 
    \item   $H_t(p)(x)=\lyap(p)(x)$ for all $t\in I$ and $(p,x)\in  R$, and
    \item the restrictions of $(p,x)\mapsto H_t(p)(x), \lyap(p)(x)$ to $(\param \times \R^n) \setminus Z$ are arbitrarily close in the compact-open $C^v$ topology for all $t\in I$.
\end{itemize}
\end{Lem}

\begin{Rem}\label{rem:fiberwise-compact-open-strong-topologies}
Equivalently, the first conclusion says that each $H_t|_{(\param \times \R^n) \setminus Z}$ is arbitrarily close to $\lyap|_{(\param \times \R^n) \setminus Z}$  in the fiberwise compact-open $C^u$ topology,  which can be defined using the space of fiberwise $u$-jets for the submersion $(\param \times \R^n) \setminus Z\to \param$ in a standard way \cite[p.~41]{cem2024hprinciple}, as can the fiberwise strong $C^u$ topology (cf. \cite[p.~59]{hirsch1976differential}).
\end{Rem}

%\iffalse
\begin{proof}
Using a closed tubular neighborhood $U\subset Q$ of $S$ \cite[p.~53]{kosinski1993differential},  a $C^\infty$ homotopy $\rho\colon I\times \param \to \param$ can be constructed so that  $\rho_0=\id_\param$, $\rho_1(U)=S$, $\rho_t|_{S\cup (\param \setminus Q)} = \id_{S\cup (\param \setminus Q)}$ for all $t\in I$,  $U_\delta\coloneqq \rho_\delta^{-1}(S)\subset Q$ is a neighborhood of $S$ for each $\delta \in (0,1]$, and $U_\delta$ is increasing in $\delta$. % ($U_{\delta_1}\subset U_{\delta_2}$ for $\delta_1 < \delta_2$).

Given $\delta \in (0,\frac{1}{2}]$, define a continuous and $C^v$ homotopy $g\colon I\times \param \to \fun^u(\R^n)$ $\rel{[S\cup (\param \setminus Q)]}$  from $\lyap$ to $\tilde{\lyap}\coloneqq \lyap\circ \rho_\delta$  by $g_t=\lyap\circ \rho_{\delta t}$.
Since $\rho_0=\id_\param$ and $\rho$ is $C^\infty$, by taking $\delta \in (0,\frac{1}{2}]$ sufficiently small we may assume that $g_t$ is arbitrarily close to $\lyap$ in the compact-open $C^v$ and fiberwise compact-open $C^u$ topologies for all $t\in I$ (cf. \cite[Ex.~2.4.10]{hirsch1976differential}).
Moreover,  $R$ is disjoint  from 
$$\tilde{Z}\coloneqq \{(p,x)\in \param \times \R^n \colon \tilde{\lyap}(p)(x)=0\}$$
by construction and, since $\param$ and hence also $Z$ are compact, by taking $\delta $ sufficiently small we may also assume that  $K \subset (\param \times \R^n)\setminus Z$ is disjoint from $\tilde{Z}$.
Note that the restriction of $g_1=\tilde{\lyap}$ to $U_{\delta}$ is a continuous and $C^s$ map $U_\delta \to \fun^r(\R^n)$.
Further note that $\tilde{Z}=Z=\param \times (\R^n\setminus \{0\})$ if $\lyap$ is $\fun_0^u(\R^n)$-valued.
For later use, fix a $C^\infty$ function $\psi\colon \param\to I$ that is $0$ on $S$ and $1$ on a neighborhood of  $\param \setminus U_{\delta}$.
 
Using the standard method of convolution with a mollifier \cite[sec.~2.2--2.3]{hirsch1976differential}, we construct a continuous map $\lyapt\colon \param \to C^0(\R^n, [0,\infty))$ such that 
\begin{itemize}
	\item  $\lyapt(p)(x)=\tilde{\lyap}(p)(x)$ for  $(p,x)\in R\cup \tilde{Z}\cup (U_\frac{\delta}{2} \times \R^n)$,
	\item  the restriction of $(p,x)\mapsto \lyapt(p)(x)$ to $(\param \times \R^n) \setminus \tilde{Z}$ is $C^s$ with $r$ continuous partial $x$-derivatives, and
	\item the restrictions of $(p,x)\mapsto \tilde{\lyap}(p)(x), \lyapt(p)(x)$ to $(\param \times \R^n) \setminus \tilde{Z}$ are arbitrarily close in the strong $C^v$ and fiberwise strong  $C^u$ topologies. 
\end{itemize}
 In particular, we can assume that
\begin{equation}\label{eq:inner-product-condition}
	\ip{\nabla \lyapt(p)(x)}{\nabla \tilde{\lyap}(p)(x)} > 0
\end{equation}
for all $(p,x)\in (\param \times \R^n) \setminus \tilde{Z}$ and that $\lyapt(p)$ is proper for all $p\in \param$.

Since $\tilde{Z}$ is compact, there is $a > 0$ such that
\begin{equation}\label{eq:k-r-subset}
K, R \subset \{(p,x)\in \param \times \R^n\colon \lyapt(p)(x) > a\}.	
\end{equation}
Given $c \in (0,\frac{a}{2})$, there is a $C^\infty$ increasing homeomorphism $\varphi\colon [0,\infty)\to [0,\infty)$ such that $\varphi^{-1}(0)=\{0\}$, $\varphi'(z)>0$ for $z > 0$,  $\varphi(z)=z$ for $z > 2c$, and  $\tilde{\lyapt}\colon \param \to C^r(\R^n, [0,\infty))$ given by $\tilde{\lyapt}(p)(x)=\varphi(\lyapt(p)(x))$ is  a well-defined continuous and $C^s$ map.\footnote{Repeat the proof of \cite[Lem.~6.3]{fathi2019smoothing} for $C^s$ functions $\param \times \R^n \to \R^n$ having $r$ continuous $x$-partial derivatives, rather than $C^\infty$ functions.}
In particular, $\tilde{\lyapt}(p)(x) = \lyapt(p)(x)$ for $(p,x)\in \{\lyap > 2c\}$, so by taking $c \in (0,\frac{a}{2})$ sufficiently small we can assume that the restrictions to $(\param \times \R^n) \setminus \tilde{Z}$ of $(p,x)\mapsto \lyapt(p)(x), \tilde{\lyapt}(p)(x)$ are arbitrarily close in the compact-open $C^s$ and fiberwise compact-open $C^r$ topologies.
Note that $\lyapt(p)(x)=\tilde{\lyapt}(p)(x)$ for all $(p,x)\in K\cup R$ by \eqref{eq:k-r-subset}, and that $\tilde{\lyapt}(p)$ is proper for each $p\in \param$ since $\lyapt(p)$ is.
Moreover, since $$\nabla \tilde{\lyapt}(p)(x)=\varphi'(\lyapt(p)(x))\nabla \lyapt(p)(x)$$ and $\varphi'(z)>0$ for $z>0$, \eqref{eq:inner-product-condition} implies that $\tilde{\lyapt}$ is $\fun^r(\R^n)$-valued.

Using the function $\psi\colon \param\to I$ fixed above, define a continuous and $C^v$ homotopy $h\colon I \times \param \to C^u(\R^n,[0,\infty))$ by
\begin{equation*}%\label{eq:final-funnel-smoothing-htpy}
h_t(p)= t\psi(p)\tilde{\lyapt}(p) + (1-t\psi(p))\tilde{\lyap}(p).
\end{equation*}
Note that  $h_1$ is a continuous and $C^s$ map $h_1\colon \param \to C^r(\R^n,[0,\infty))$ since $\tilde{\lyapt}$ and $\tilde{\lyap}|_{U_\delta}\colon U_\delta \to \fun^r(\R^n)$ are continuous and $C^s$ maps and $\psi$ is $1$ on a neighborhood of $\param \setminus U_{\delta}$.
And properness of each $\tilde{\lyap}(p), \tilde{\lyapt}(p)$ together with the properties of $\psi$, \eqref{eq:inner-product-condition}, and linearity imply that $h$ is in fact an $\fun^u(\R^n)$-valued map $h\colon \param \to \fun^u(\R^n)$, so $h_1$ is a continuous and $C^s$ map $h_1\colon \param \to \fun^r(\R^n)$.
Moreover, $h$ is a homotopy $\rel{S}$ since $\psi$ is $0$ on $S$, and $h_t(p)(x)=\lyap(p)(x)$ for all $t\in I$ and $(p,x)\in R$ since $$\tilde{\lyapt}(p)(x)= \tilde{\lyap}(p)(x) = \lyap(p)(x)$$ for $(p,x)\in R$.
Furthermore, $h$ is $\fun_0^u(\R^n)$-valued if $\lyap$ is since then, as noted above,  $\tilde{Z}=Z=\param \times (\R^n\setminus \{0\})$.
Finally, by the approximation properties of $\tilde{\lyapt}$ and $\lyapt$, we may assume that the restrictions to $(\param \times \R^n) \setminus \tilde{Z}$ of $(p,x)\mapsto h_t(p)(x), \tilde{\lyap}(p)(x)$ are arbitrarily close in the compact-open $C^v$ and fiberwise compact-open $C^u$ topologies for all $t\in I$ (cf. \cite[pp.~47--48]{hirsch1976differential}).

Thus, smoothly concatenating (as in \eqref{eq:gamma-def}) the homotopies $g$ from $\lyap$ to $\tilde{\lyap}$ and $h$ from $\tilde{\lyap}$ to $h_1$ yields a continuous and $C^v$  homotopy $H\colon I\times \param \to \fun^u(\R^n)$ from $\lyap$ to the continuous and $C^s$ map $H_1=h_1\colon \param \to \fun^r(\R^n)$ having the desired properties.
\end{proof}
%\fi

Recall that a continuous map is a \concept{weak homotopy equivalence} \cite[Def.~VII.11.11]{bredon1993topology} if it induces a bijection on path components and isomorphisms on all homotopy groups \cite[Sec.~VII.4]{bredon1993topology}.

\begin{Co}\label{co:whe-funnels}
	Fix $r,s\in \N_{\geq 1}\cup\{\infty\}$ with $r > s$.
	The inclusion maps $\fun_0^r(\R^n) \hookrightarrow \fun_0^s(\R^n)$ and $\fun^r(\R^n) \hookrightarrow \fun^s(\R^n)$ are weak homotopy equivalences.
\end{Co}

%\iffalse
\begin{proof}
For each $k\in \N_{\geq 0}$,  the homomorphisms of $k$-th homotopy groups induced by the inclusions are surjective by Lemma~\ref{lem:smoothing-funnels} with $(\param,R,S)=(\sph^k,\varnothing, \{*\})$ and $(r,s,u,v)$ replaced by $(r,0,s,0)$, and injective by Lemma~\ref{lem:smoothing-funnels} with $(\param, R, S)=(D^{k+1}, \varnothing, \sph^k)$ and $(r,s,u,v)$ replaced by $(r,0,s,0)$.

\end{proof}
%\fi

\subsection{Contractibility and path-connectedness lemmas}\label{subsec:contractibility-path-conn}

We will use the following consequence of Lemmas~\ref{lem:funnel-alexander}, \ref{lem:smoothing-funnels} to prove weak contractibility in Theorem~\ref{th:funnel-space-fiber-bundle-over-U}.
Recall that a topological space is \concept{weakly contractible} if it is path-connected and all of its homotopy groups are trivial.

\begin{Lem}\label{lem:contractible-fibers}
	The subspace $\{\lyap \in \fun_0^\infty(\R^n) \colon \lyap^{-1}([0,1])=D^n\}$ is weakly contractible.	
\end{Lem}

%\iffalse
\begin{proof}
	Denote by $\F= \{\lyap \in \fun_0^\infty(\R^n) \colon \lyap^{-1}([0,1])=D^n\}$ those functions with 1-sublevel set the standard $n$-disc.	
	
	Let $\param$ be a compact $C^\infty$ manifold and $\lyap\colon \param \to \F$ be a continuous map.
	It suffices to show that $\lyap$ is nullhomotopic.

	Let $\I\subset C^\infty(I\times \R^n, \R^n)$ be the space of isotopies of $\R^n$ starting at the identity that are stationary on $\sph^{n-1}$ and on the complement of a fixed compact neighborhood $U\subset \R^n\setminus \{0\}$  of $\sph^{n-1}$. 
	Since $\param$ is compact there is $\varepsilon > 0$ such that, for all $p\in \param$, the gradient flow of $\lyap(p)$ induces a closed tubular neighborhood $\sph^{n-1}\times [-\varepsilon,\varepsilon]\hookrightarrow \interior{U}$ of $\sph^{n-1}$ which, by a standard fact from ordinary differential equations \cite[Thm~B.3]{duistermaat2000lie}, depends continuously on $p$ in the compact-open $C^\infty$ topology.
	Since the standard proof of isotopy uniqueness of tubular neighborhoods \cite[Thm~III.3.5]{kosinski1993differential} works parametrically by simply inserting everywhere a parameter ``$p$'', there is a continuous map $\Phi\colon \param \to \I$ such that $\Phi(p)_0=\id_{\R^n}$ and $\Phi(p)_1(\partial B_{\sqrt{c}}(0))=\lyap(p)^{-1}(c)$ for all $c>0$ sufficiently close to $1$.

	Thus, $(t,p)\mapsto \lyap(p)\circ \Phi(p)_t$ defines a homotopy $I\times \param \to \F$ from $\lyap$ to a continuous map $\lyap'\colon \param \to \F$ such that $\lyap'(p)(x)=\norm{x}^2$ for all $(p,x)$ in an open neighborhood of $\param \times \sph^{n-1}$.
	Let $R\subset \param \times (\R^n\setminus \{0\})$ be a smaller compact neighborhood of $\param \times \sph^{n-1}$.
	The first conclusion of Lemma~\ref{lem:smoothing-funnels} (with $S=\varnothing$) then produces a continuous homotopy $I\times \param \to \F$ from $\lyap'$ to a $C^\infty$ map $\lyap''\colon \param \to \F$ such that $\lyap''(p)(x)=\norm{x}^2$ for all $(p,x)$ in $R$. 
	
	Replacing $\lyap''(p)(x)$ with $\norm{x}^2$ outside $\param \times D^n$ then yields a $C^\infty$ map $J\colon \param \to \F$ such that $J(p)(x)=\lyap''(p)(x)$ for all $(p,x)$ in a neighborhood of $\param \times \{0\}$ and $J(p)(x)= \norm{x}^2$ for all $(p,x)$ outside a neighborhood of $\param \times \{0\}$.
	Lemma~\ref{lem:funnel-alexander} then furnishes a $C^\infty$ homotopy $I\times \param \to \F$ from $\lyap''$ to the constant map sending all of $\param$ to $x\mapsto \norm{x}^2$.
	Finally, concatenating the above homotopies $I\times \param \to \F$ from $\lyap$ to $\lyap'$, $\lyap'$ to $\lyap''$, and $\lyap''$ to the constant map $p\mapsto (x\mapsto \norm{x}^2)$ gives the desired nullhomotopy of $\lyap\colon \param \to \F$.
\end{proof}
%\fi

The following is a version of a result of Palais \cite{palais1960extending} (see also Cerf \cite{cerf1961topologie}) on isotopy uniqueness of $C^\infty$ embeddings of discs \cite[Thm~III.3.6]{kosinski1993differential}.

\begin{Lem}\label{lem:emb-isotopy-pres-0}
	Let $f,g\colon D^n\to \R^n$ be orientation-preserving $C^\infty$ embeddings satisfying $f(0)=0=g(0)$.
	Then there is a $C^\infty$ diffeotopy $h\colon I\times \R^n\to \R^n$ satisfying $h_0=\id_{\R^n}$, $h_1\circ f = g$, and $h_t(0)=0$ for all $t\in I$.
\end{Lem}	

In particular, the following corollary will be used in the proof of Theorem~\ref{th:funnel-space-fiber-bundle-over-U}.
(See sections~\ref{subsubsection:diffeos-embeddings}, \ref{subsubsection:nonlinear-grass} for the definitions of $\U_0(D^n,\R^n)$ and $\Emb_0^+(D^n,\R^n)$.)

\begin{Co}\label{co:nonlinear-grassmannian-path-conn}
$\U_0(D^n,\R^n)$ is path-connected for any $n$.
\end{Co}

%\iffalse
\begin{proof}
Fix any $M, N\in \U_0(D^n,\R^n)$.
There are $f,g\in \Emb_{0}^{+}(D^n,\R^n)$ satisfying $f(D^n)=M$, $g(D^n)=N$ and additionally $f(0)=0=g(0)$.
A diffeotopy as in Lemma~\ref{lem:emb-isotopy-pres-0} yields in particular a path from $f$ to $g$ in $\Emb_{0}^{+}(D^n,\R^n)$, which the quotient map  $\Emb_{0}^{+}(D^n,\R^n)\to \U_0(D^n,\R^n)$ sends to a path from $M$ to $N$. 
\end{proof}	
%\fi

\subsection{Fiber bundle structure and weak homotopy type}\label{subsec:fiber-bundle-structure-weak-homotopy-type}

The following lemma justifies the definition below.

\begin{Lem}\label{lem:h-cobordism}
	Let $n\neq 4,5$ and $\lyap,\lyapt\in \fun_0^\infty(\R^n)$.
	Fix  $a, b >0$ be such that $\{\lyap \leq a\} \subset \{\lyapt < b\}$. 
	Then $\{\lyap \leq a\}$ is diffeomorphic to $D^n$ and $\{a \leq \lyap\} \cap \{\lyapt \leq b\}$ is diffeomorphic to $\sph^{n-1}\times I$.
\end{Lem}

%\iffalse
\begin{proof}
	Note that $\R^n\setminus \{0\}$ deformation retracts onto each of $\{\lyap = a\}$, $\{\lyapt =b\}$, and $\{a \leq \lyap\} \cap \{\lyapt \leq b\}$ by moving along trajectories of $\nabla \lyap$ and $\nabla \lyapt$ (cf. \cite[p.~325]{wilson1967structure}).
	Thus, the inclusion maps of $\{\lyap = a\}$, $\{\lyapt =b\}$, and $\{a \leq \lyap\} \cap \{\lyapt \leq b\}$ into $\R^n\setminus \{0\}$ are each homotopy equivalences.
	This implies that the inclusion maps of $\{\lyap = a\}$ and $\{\lyapt =b\}$ into $\{a \leq \lyap\} \cap  \{\lyapt \leq b\}$ are also homotopy equivalences, so $\{a \leq \lyap\}\cap  \{ \lyapt \leq b\}$ is an h-cobordism.
	Similarly, by moving along trajectories of $\nabla \lyap$ we see that $\{\lyap \leq a\}$ is a compact contractible manifold with boundary homotopy equivalent to $\sph^{n-1}\simeq \R^n\setminus \{0\}$.
	
	Since $n\neq 4,5$, Smale's h-cobordism theorem (\cite[Cor.~3.2, Thm~5.1]{smale1962structure}, \cite[Thm~9.1, Prop.~A]{milnor1965hcobordism}), Perelman's proof of the Poincar\'{e} conjecture \cite{perelman2002entropy,perelman2003ricci,perelman2003finite} (see \cite[Cor.~0.2]{morgan2007ricci}), and the classifications of manifolds of dimension $2$  \cite[Thm~9.3.11]{hirsch1976differential} and $1$ \cite[Appendix]{milnor1965topology} imply that $\{\lyap \leq a\}$ is diffeomorphic to $D^n$ and  $\{a\leq \lyap\} \cap \{\lyapt \leq b\}$ is diffeomorphic to $\{V=a\}\times I\approx \partial D^n \times I=\sph^{n-1}\times I$.
\end{proof}
%\fi

For $n\neq 4,5$ and any $\lyap\in \fun_0^\infty(\R^n)$, the sublevel set $\lyap^{-1}([0,1])$ is diffeomorphic to $D^n$ by Lemma~\ref{lem:h-cobordism}.
Thus, the following definition makes sense.

\begin{Def}\label{def:sublevel-set-map}
	For $n\neq 4,5$, define the \concept{sublevel set map} by $$p\colon \fun_0^\infty(\R^n)\to \U_0(D^n,\R^n), \quad p(\lyap)=\lyap^{-1}([0,1]).$$
\end{Def}

A theorem of Abraham on smoothness of evaluation maps \cite[Thm~10.3]{abraham1967transversal} and the implicit function theorem of Hildebrandt and Graves \cite[Thm~4]{hildebrandt1927implicit} imply that the sublevel set map is continuous, and it is also surjective by the next lemma.

\begin{Lem}\label{lem:surjectivity-sublevel-set-map}
	For $n\neq 4,5$, the sublevel set map is surjective.
\end{Lem}

%\iffalse
\begin{proof}
Fix any $M\in \U_0(D^n,\R^n)$ and $C^\infty$ orientation-preserving embedding $f\colon D^n \hookrightarrow \R^n$ satisfying $f(D^n)=M$ and $f(0)=0$.
Lemma~\ref{lem:emb-isotopy-pres-0} furnishes a $C^\infty$ diffeotopy $h\colon I\times \R^n\to \R^n$ satisfying $h_0=\id_{\R^n}$ and $h_1|_{D^n}=f$.
Defining $\lyap, \lyapt\in \fun_0^\infty(\R^n)$ by $\lyapt(x)=\norm{x}^2$ and $\lyap=\lyapt\circ h_1^{-1}$, we see that $p(\lyap)=h_1(\lyapt^{-1}([0,1]))=h_1(D^n)=f(D^n)=M$.
\end{proof}
%\fi

The following is one of our main results, and it completes the proof of Theorem~\ref{th:intro-grassmannian}.

\begin{Th}\label{th:funnel-space-fiber-bundle-over-U}
	For $n\neq 4,5$ the sublevel set map $p\colon \fun_0^\infty(\R^n)\to \U_0(D^n,\R^n)$ is a topological fiber bundle with weakly contractible fibers.
	In particular, $p$ is a weak homotopy equivalence.
\end{Th}

\begin{Rem}\label{rem:prelim-path-conn}
	Theorem~\ref{th:funnel-space-fiber-bundle-over-U} and Corollary~\ref{co:nonlinear-grassmannian-path-conn} directly imply that $\fun_0^\infty(\R^n)$ is path-connected.
	Alternatively, Lemmas~\ref{lem:funnel-alexander}, \ref{lem:h-cobordism} more directly imply path-connectedness and moreover that there is a $C^\infty$ path connecting any $\lyap, \lyapt\in  \fun_0^\infty(\R^n)$.
    These results are later stated as part of a stronger one (Corollary~\ref{co:funnels-trivial-homotopy-groups}).
\end{Rem}

%\iffalse
\begin{proof}
The sublevel set map is surjective by Lemma~\ref{lem:surjectivity-sublevel-set-map}.
Fix any $M\in \U_0(D^n,\R^n)$ and set $\F\coloneqq p^{-1}(M)\subset \fun_0^\infty(\R^n)$.
Since $\U_0(D^n,\R^n)$ is path-connected (Corollary~\ref{co:nonlinear-grassmannian-path-conn}), it suffices to find an open neighborhood  $U\subset \U_0(D^n,\R^n)$ of $M$ and a continuous map $f\colon p^{-1}(U)\to \F$ such that $(p,f)\colon p^{-1}(U)\to U \times \F$ is a homeomorphism.

Let $\pi\colon E\subset \R^n\setminus \{0\}\to M$ be a tubular neighborhood of $\partial M$, so $\pi\colon E\to M$ is a vector bundle with zero section identified with $\partial M$.
Fix nested open neighborhoods $E_1,E_2, E_3 \subset E$ of  $\partial M$ satisfying $\cl{E_1}\subset E_2$ and $\cl{E_2}\subset E_3$.
Fix an open neighborhood $U\subset \U_0(D^n,\R^n)$ of $M$ such that the boundary $\partial N$ of every $N\in U$ is the image of a unique $C^\infty$ section $\sigma_N\colon \partial M\to E_1\subset E$ of $\pi$ that depends continuously on $N$ with respect to the compact-open $C^\infty$ topology.

Let $\varphi\in C^\infty(E,I)$ be $1$ on $E_2$ and $0$ on a neighborhood of $E\ct E_3$. 
Define a continuous map $\vo\colon U\to \vf^\infty(\R^n)$ by $\vo(N)(x)=0$ for $x\not \in E_3$, and for $x\in E_3$ by 
\begin{equation*}
\vo(N)(x)= \varphi(x)\vl_x(\sigma_N(\pi(x))),
\end{equation*}
where $\vl_{x_m}\colon E_m\to V_{x_m} E$ is the vertical lift of $E_m=\pi^{-1}(m)$ to the vertical space $V_{x_m} E \subset T_{x_m}E$ tangent to the fiber $E_m$ \cite[p.~107]{michor2008topics}.

Since $\vo(N)$ vanishes outside of a compact subset of $\R^n\setminus \{0\}$ for each $N\in U$, there is a continuous map $\Psi\colon U \to \Diff_0(\R^n)$ with $\Psi(N)$ given by the time-$1$ map of the flow of $\vo(N)$ \cite[Thm~B.3]{duistermaat2000lie}, where $\Diff_0(\R^n)\subset \Diff(\R^n)$ denotes the diffeomorphisms fixing the origin.
By construction, $\Psi(N)(\partial M)=\partial N$ and hence also 
\begin{equation*}
\Psi(N)(M)=N \quad 	\textnormal{and} \quad M=\Psi(N)^{-1}(N) 
\end{equation*}
for all $N\in U$.

Finally, define a continuous map $f\colon p^{-1}(U)\to \F\subset \fun_0^\infty(\R^n)$ by $f(\lyap)=\lyap \circ \Psi(p(\lyap))$.
Since 
\begin{align*}
f(\lyap)^{-1}([0,1])&=\Psi(p(\lyap))^{-1}(\lyap^{-1}([0,1])) =	\Psi(p(\lyap))^{-1}(p(\lyap)) \\
&= M,
\end{align*}
we see that $f$ is indeed $\F$-valued.
And $(p,f)\colon p^{-1}(U)\to U \times \F$ is a homeomorphism with continuous inverse $g\colon U\times \F \to p^{-1}(U)$ given by $g(N,\lyapt)=\lyapt \circ \Psi(N)^{-1}$, since
\begin{align*}
	g\circ (p,f)(\lyap)&= g(p(\lyap),f(\lyap))= f(\lyap)\circ \Psi(p(\lyap))^{-1}= \lyap \circ \Psi(p(\lyap)) \circ \Psi(p(\lyap))^{-1}  \\
	&=\lyap,
\end{align*}
\begin{align*}
p\circ g(N,\lyapt)&=g(N,W)^{-1}([0,1])=\Psi(N)(\lyapt^{-1}([0,1]))=\Psi(N)(M)\\
&= N,	
\end{align*}
and hence
\begin{align*}
	f\circ g(N,\lyapt)&= g(N,\lyapt)\circ \Psi(p(g(N,\lyapt))) = g(N,\lyapt)\circ \Psi(N)\\
	&= \lyapt \circ \Psi(N)^{-1}\circ \Psi(N)\\
	&= \lyapt.
\end{align*}
Weak contractibility of the fibers of $p$ then follows from Lemma~\ref{lem:contractible-fibers} and the fact that all fibers of a topological fiber bundle are homeomorphic.
Since any fiber bundle is a Serre fibration, $p$ is a weak homotopy equivalence by the corresponding long exact sequence of homotopy groups \cite[Thm~VII.6.7]{bredon1993topology}.
\end{proof}
%\fi

\subsection{Homotopy groups and smooth extensions}\label{subsec:vanishing-htpy-groups-funnels}

With Theorem~\ref{th:funnel-space-fiber-bundle-over-U} in hand, we begin an investigation of the topologies of $\U_0(D^n,\R^n)$ and $\U(D^n,\R^n)$  culminating in Proposition~\ref{prop:nonlinear-grass-trivial-htpy-groups}.
The corollaries that follow then complete the proof of Theorems~\ref{th:intro-1-conn-contract}, \ref{th:intro-cr-extensions} for $\fun_0^r(\R^n)$.

\begin{Lem}\label{lem:deriv-eplus0-gl-n-htpy-equiv}
For any $n$ and $y\in D^n$, the ``evaluate derivative at $y$'' map 
\begin{equation}\label{eq:lem:deriv-eplus0-gl-n-htpy-equiv}
	\Emb_{0}^{+}(D^n,\R^n)\to \GL^+(n), \qquad f\mapsto D_y f
\end{equation}
is the time-1 map of a strong deformation retraction, and the inclusion
\begin{equation}\label{eq:lem:deriv-eplus0-gl-n-htpy-equiv-2}
	\Emb_{0}^{+}(D^n,\R^n)\hookrightarrow \Emb^{+}(D^n,\R^n) 
\end{equation}
is a homotopy equivalence.
\end{Lem}	

%\iffalse
\begin{proof}
 The map $z\colon \Emb_{0}^{+}(D^n,\R^n)\to \interior{D^n}$ defined by $z(f)=f^{-1}(0)$ is clearly continuous, so $h\colon I\times \Emb_{0}^{+}(D^n,\R^n) \to \Emb_{0}^{+}(D^n,\R^n)$ defined by  
\begin{equation*}
h_t(f)(x)=	
    \begin{cases}
	 \frac{1}{1-2t}f(z(f) + (x-z(f))(1-2t)),& 0\leq t < \frac{1}{2}\\
	 D_{(2-2t)z(f)+(2t-1)y}f(x+(2t-2)z(f)), & \frac{1}{2}\leq t \leq 1
	\end{cases}
\end{equation*}
is  continuous by Hadamard's lemma.
And $h$ is indeed $\Emb_{0}^{+}(D^n,\R^n)$-valued since, for each $t\in I$ and $f\in \Emb_{0}^{+}(D^n,\R^n)$, (i) $h_t(f)$ is either an orientation-preserving invertible affine map or a composition with $f$ thereof, and (ii)  $h_t(f)^{-1}(0)$ lies on the line segment in $\interior{D^n}$ with endpoints $0$ and $z(f)$.
Since $h_0(f)=f$,  $h_1(f)=D_y f\in \GL^+(n)$, and $h_t|_{\GL^+(n)}=\id_{\GL^+(n)}$ for all $t\in I$, $h$ is the desired strong deformation retraction of $\Emb_{0}^{+}(D^n,\R^n)$ onto $\GL(n)$.

Next, in the commutative diagram
\begin{equation}\label{eq:embedding-triangle}
	\begin{tikzcd}
		\Emb_0^+(D^n,\R^n) \arrow[hookrightarrow]{rr} \arrow["D_0",swap]{dr} && \Emb^+(D^n,\R^n) \arrow["D_0"]{dl}\\
		& \GL^+(n) & 
	\end{tikzcd} ,
\end{equation}
the diagonal arrows are homotopy equivalences by the above for the left side, and by a standard and simpler ``zoom in at $0$'' argument for the right side  \cite[Thm~9.1.1, Cor.~9.1.4]{kupers2019lectures}.
Thus, the horizontal arrow \eqref{eq:lem:deriv-eplus0-gl-n-htpy-equiv-2} in \eqref{eq:embedding-triangle} is also a homotopy equivalence.
\end{proof}
%\fi

\begin{Lem}\label{lem:gr-0-gr-w-h-e}
	The inclusion
	\begin{equation}\label{eq:lem:gr-0-gr-w-h-e}
	\U_0(D^n,\R^n)\hookrightarrow \U(D^n,\R^n)	
	\end{equation}
	is a weak homotopy equivalence for any $n$.
\end{Lem}

\begin{proof}
We have a commutative diagram
\begin{equation*}%\label{eq:lem:gr-0-gr-w-h-e-2}
	\begin{tikzcd}
		\Emb_0^+(D^n,\R^n) \arrow[hookrightarrow]{r} \arrow{d} & \Emb^+(D^n,\R^n) \arrow{d}\\
		\U_0(D^n,\R^n) \arrow[hookrightarrow]{r} & \U(D^n,\R^n)
	\end{tikzcd}
\end{equation*}
in which the horizontal arrows are inclusions and the vertical arrows are the fiber bundle projections  $f\mapsto f(D^n)$ \cite[Thm~2.2]{gaybalmaz2014principal}.
Since the induced map on the fibers $\Diff^+(D^n)$ over the standard $n$-disc is the identity and the top horizontal arrow is a homotopy equivalence by Lemma~\ref{lem:deriv-eplus0-gl-n-htpy-equiv}, it follows that \eqref{eq:lem:gr-0-gr-w-h-e} is a weak homotopy equivalence \cite[Prop.~4.7]{mitchell2001notes}.
\end{proof}

The injectivity portion of the next lemma is a known consequence of results of Smale \cite{smale1959diffeomorphisms} and Hatcher \cite{hatcher1983proof}. 
Here $\pi_k X$ denotes the $k$-th homotopy group of a topological space $X$ for $k\geq 1$ and the set of path components of $X$ for $k=0$ \cite[Sec.~VII.4]{bredon1993topology}.

\begin{Lem}\label{lem:smale-hatcher-eval-deriv-htpy-equiv}
	For each $k\in \N_{\geq 0}$ and $y\in D^n$, the ``evaluate derivative at $y$''-induced map $$\pi_k\Diff^+(D^n)\to \pi_k\GL^+(n)$$ is surjective for any $n$ and bijective if $n\leq 3$.
\end{Lem}

%\iffalse
\begin{proof}
Let $\GS\colon \GL(n)\to SO(n)$ be the result of applying the Gram-Schmidt process starting from the left column.
We have a commutative diagram
\begin{equation*}
	\begin{tikzcd}
		& \Diff^+(D^n) \arrow["D_y"]{d}\\
		& \GL^+(n)\arrow["\GS"]{d}\\
		SO(n) \arrow[hookrightarrow]{uur}\arrow[equal]{r}& SO(n)
	\end{tikzcd},
\end{equation*}	
so $\GS\circ D_y$ induces a surjection on all $\pi_k$.
Since $\GS$ is a homotopy equivalence, the ``evaluate derivative at $y$'' map $D_y$ also induces a surjection on all $\pi_k$. 	

Assume now that $n\leq 3$.
Let $(e_1,\ldots, e_n)$ be the standard orthonormal basis of $\R^n$.
Since the homotopy class of $D_y$  is independent of $y\in D^n$ we may assume that $y=e_1\in \partial D^n$, and  $D_{e_1}$ is clearly  homotopic to the composition
\begin{equation*}
	\begin{tikzcd}
		\Diff^+(D^n) \arrow["\rho"]{r}& \Diff^+(\sph^{n-1}) \arrow["f"]{r} & \GL^+(n)
	\end{tikzcd}
\end{equation*}
of the restriction map $\rho$ and the map $f$ given by adjoining the value and derivative at $e_1$.
Thus, it suffices to show that both $\rho$ and $f$ are weak homotopy equivalences. 

By Cerf's first fibration theorem, $\rho$ is a fiber bundle with fiber $\rho^{-1}(\id_{\sph^{n-1}})$ equal to those diffeomorphisms $\Diff_\partial(D^n)\subset \Diff^+(D^n)$ that restrict to the identity on $\partial D^n = \sph^{n-1}$ \cite[sec.~2.2.2, Cor.~2]{cerf1961topologie}.
This fiber is weakly contractible for (i) $n=1$ by a convexity argument \cite[Thm~3.2.1]{kupers2019lectures}, (ii) $n=2$ by a theorem of Smale (\cite[Thm~B]{smale1959diffeomorphisms}, \cite[sec.~7.1]{kupers2019lectures}), and (iii) $n=3$ by Hatcher's proof of the Smale conjecture \cite[p.~604]{hatcher1983proof}.
Thus, $\rho$ is a weak homotopy equivalence by the long exact sequence of homotopy groups for a fibration \cite[Thm~VII.6.7]{bredon1993topology}.

To complete the proof it remains only to show that $f$ is a weak homotopy equivalence. 
This is trivial for $n=1$, so it suffices to consider $n=2,3$.
Making the natural identification of $\GL^+(n)$ with the bundle $\Fr^{+}(T\sph^{n-1})$ of positively oriented $(n-1)$-frames of vectors tangent to $\sph^{n-1}$, we see that $f$ factors as the composition
\begin{equation}\label{eq:restriction-composition}
	\begin{tikzcd}
		\Diff^+(\sph^{n-1}) \arrow{r}
		& \Emb^+(D^{n-1}_+,\sph^{n-1}) \arrow["\simeq"]{r}
		\arrow[d, phantom, ""{coordinate, name=Z}]
		& \Emb^+(\interior{D^{n-1}_+}, \sph^{n-1}) \arrow[dll,
		rounded corners,
		"\simeq",
		to path={ -- ([xshift=2ex]\tikztostart.east)
			|- (Z) [near end]\tikztonodes
			-| ([xshift=-2ex]\tikztotarget.west)
			-- (\tikztotarget)}] \\
		\Fr^+(T \sph^{n-1})
		& ~
		& 
	\end{tikzcd}
\end{equation} 
in which $D^{n-1}_+$ denotes the hemisphere of $\sph^{n-1}$ centered at $e_{1}$, the first two arrows are restriction maps, and the long arrow is given by taking the value and derivative at $e_1$.
The long arrow is well-known to be a weak homotopy equivalence \cite[Thm~9.1.2]{kupers2019lectures}, and the second restriction is a homotopy equivalence since there are $C^\infty$ embeddings $D^{n-1}_+\hookrightarrow \interior{D^{n-1}_+}$ and $\interior{D^{n-1}_+}\hookrightarrow D^{n-1}_+ $ that are inverse up to isotopy \cite[Lem.~9.1.3]{kupers2019lectures}.
Thus, it suffices to prove that the first restriction is a weak homotopy equivalence.
That restriction is a fiber bundle by Cerf's first fibration theorem \cite[sec.~2.2.2, Cor.~2]{cerf1961topologie}, so by the long exact sequence of homotopy groups it suffices to observe that its fiber 
\begin{equation*}
	\Diff(\sph^{n-1} \text{ rel } D^{n-1}_+) \cong \Diff_\partial(D^{n-1})
\end{equation*}
over the inclusion map is weakly contractible by the aforementioned convexity argument for $n=2$ and theorem of Smale for $n=3$.
\end{proof}
%\fi

\begin{Lem}\label{lem:diffplus-eplus0-inclusion-htpy-epi-iso}
	For each $k\in \N_{\geq 0}$, the inclusion-induced map $$\pi_k \Diff^+(D^n)\to \pi_k\Emb^{+}(D^n,\R^n)$$ is surjective for any $n$ and bijective if $n\leq 3$. 
\end{Lem}

%\iffalse
\begin{proof}
We have a commutative diagram
\begin{equation}\label{eq:lem:diffplus-eplus0-inclusion-htpy-epi-iso}
	\begin{tikzcd}
		\Diff^+(D^n) \arrow[hookrightarrow, "i"]{rr} \arrow["\text{surjective on }\pi_k", swap]{dr} & & \Emb^{+}(D^n,\R^n) \arrow["\simeq"]{dl}\\
		& \GL^+(n) &
	\end{tikzcd}
\end{equation}
where $i$ is the inclusion and the diagonal arrows are ``evaluate derivative at a point''.
The left arrow is surjective on all $\pi_k$ by Lemma \ref{lem:smale-hatcher-eval-deriv-htpy-equiv}, and the right arrow is bijective on all $\pi_k$ by the standard ``zoom in at $0$'' argument mentioned in the proof of Lemma~\ref{lem:deriv-eplus0-gl-n-htpy-equiv} \cite[Thm~9.1.1, Cor.~9.1.4]{kupers2019lectures}.
Lemma~\ref{lem:smale-hatcher-eval-deriv-htpy-equiv} also implies that, if $n\leq 3$, the left diagonal arrow is bijective on all $\pi_k$.
Since \eqref{eq:lem:diffplus-eplus0-inclusion-htpy-epi-iso} commutes, this implies that $i$ induces a surjection on all $\pi_k$ for any $n$ and a bijection if $n\leq 3$.
\end{proof}	
%\fi

In addition to the results of Smale and Hatcher contained in Lemma~\ref{lem:smale-hatcher-eval-deriv-htpy-equiv}, the proof of the following proposition relies on a result of Cerf \cite{cerf1970stratification}. 
Recall that a topological space is \concept{simply connected} if it is path-connected and its first homotopy group is trivial.
(We choose to be explicit by redundantly stating that a simply-connected space is path-connected in various results.)

\begin{Prop}\label{prop:nonlinear-grass-trivial-htpy-groups}
$\U(D^n,\R^n)$ is path-connected for all $n$, simply connected if $n\neq 4,5$, and weakly contractible if $n\leq 3$.
\end{Prop}

%\iffalse
\begin{proof}
The path-connectedness statement follows from Corollary~\ref{co:nonlinear-grassmannian-path-conn} and Lemma~\ref{lem:gr-0-gr-w-h-e}.

Next, the quotient map $\Emb^{+}(D^n,\R^n)\to \U(D^n,\R^n)$ is a fiber bundle with fiber $\Diff^+(D^n)$ over $D^n\in \U(D^n,\R^n)$  \cite[Thm~2.2]{gaybalmaz2014principal}.
The long exact sequence of homotopy groups of this fibration \cite[Thm~VII.6.7]{bredon1993topology} contains the segment
\begin{equation}\label{eq:cd-les-snake}
 \begin{tikzcd}
 	\pi_k \Diff^+(D^n) \arrow[r,"\text{surjective}"]
 	& \pi_k \Emb^{+}(D^n,\R^n) \arrow[r,"0"]
 	\arrow[d, phantom, ""{coordinate, name=Z}]
 	& \pi_k \U(D^n, \R^n) \arrow[dll,
 	rounded corners,
 	"\text{injective}",
 	to path={ -- ([xshift=2ex]\tikztostart.east)
 		|- (Z) [near end]\tikztonodes
 		-| ([xshift=-2ex]\tikztotarget.west)
 		-- (\tikztotarget)}] \\
 		\pi_{k-1} \Diff^+(D^n) \arrow[r, "\text{surjective}"]
 	& \pi_{k-1} \Emb^{+}(D^n,\R^n)
 	& 
 \end{tikzcd}
\end{equation} 
in which the indicated arrows are surjective by Lemma~\ref{lem:diffplus-eplus0-inclusion-htpy-epi-iso}.
Exactness then implies that the remaining short arrow is zero and hence the long arrow is injective.
 
Cerf's pseudoisotopy theorem implies that $\pi_0 \Diff^+(D^n)=\{*\}$ for $n\geq 6$ \cite[p.~8, Cor.~1]{cerf1970stratification}, so exactness of \eqref{eq:cd-les-snake} yields $\pi_1 \U(D^n,\R^n)=\{*\}$ for $n\geq 6$.

If $n\leq 3$, the surjections in \eqref{eq:cd-les-snake} are bijections by Lemma~\ref{lem:diffplus-eplus0-inclusion-htpy-epi-iso}, so exactness of \eqref{eq:cd-les-snake} implies that the long injective arrow is zero.
Exactness then implies that $\pi_k \U(D^n,\R^n)=\{*\}$ for all $k\in \N_{\geq 1}$ if $n\leq 3$.
This completes the proof.
\end{proof}
%\fi

The following pair of corollaries are two of our main results.
They complete the proof of Theorems~\ref{th:intro-1-conn-contract}, \ref{th:intro-cr-extensions} for $\fun_0^r(\R^n)$.

Recall that topological spaces $X$ and $Y$ have the same \concept{weak homotopy type} if they are equivalent under the equivalence relation generated by the non-symmetric relation of weak homotopy equivalence.

\begin{Co}\label{co:funnels-homotopy-groups}
	For any $r\in \N_{\geq 1}\cup \{\infty\}$ and $n\neq 4,5$, $\fun_0^r(\R^n)$ has the weak homotopy type of $\U(D^n,\R^n)$. In particular, for each $k\in \N_{\geq 0}$ there is an isomorphism $$\pi_k \fun_0^r(\R^n)\cong \pi_k \U(D^n,\R^n).$$ 	
\end{Co}

%\iffalse
\begin{proof}
	This is immediate from Theorem~\ref{th:funnel-space-fiber-bundle-over-U}, Corollary~\ref{co:whe-funnels}, and Lemma~\ref{lem:gr-0-gr-w-h-e}.
\end{proof}
%\fi

\begin{Co}\label{co:funnels-trivial-homotopy-groups}
For any $r\in \N_{\geq 1} \cup \{\infty\}$,  $\fun_0^r(\R^n)$  is path-connected and simply connected if $n\neq 4,5$ and weakly contractible if $n\leq 3$.
In fact, if $\param$ is a compact $C^\infty$ manifold with boundary, $S\subset \param$ is a closed subset and $C^\infty$ neatly embedded submanifold of either $\param$ or $\partial \param$, and $\lyap_S\colon S \to \fun_0^r(\R^n)$ is continuous and $C^s$ with $0 \leq s\leq r$, then there is a continuous and $C^s$ map $\lyap\colon \param \to \fun_0^r(\R^n)$ extending $\lyap_S$ if either (i) $n\neq 4,5$ and $\dim \param \leq 2$ or (ii) $n\leq 3$.
\end{Co}

%\iffalse
\begin{proof}
Corollary~\ref{co:funnels-homotopy-groups}	and Proposition~\ref{prop:nonlinear-grass-trivial-htpy-groups} directly imply the first statement.
The existence of a continuous map $\tilde{\lyap}\colon \param \to \fun_0^r(\R^n)$ equal to $\lyap_S$ on $S$ then follows from the first statement and obstruction theory \cite[Cor.~VII.13.13]{bredon1993topology}.
Lemma~\ref{lem:smoothing-funnels} with $(r,s,u,v)$ replaced by $(r,s,r,0)$  then furnishes a continuous and $C^s$  map $\lyap\colon \param \to \fun_0^r(\R^n)$ equal to $\lyap_S$ on $S$.
\end{proof}
%\fi

\section{Topology of spaces of asymptotically stable vector fields}\label{sec:htpy-groups-gas-vf}

This section proves  Theorems~\ref{th:intro-1-conn-contract}, \ref{th:intro-cr-extensions}, \ref{th:intro-weak-htpy-vfs-etc} for $\stab_0^r(\R^n)$ and $\stabc_0^r(\R^n)$.
Section~\ref{subsec:smoothing-to-gradient-htpy-groups} contains a ``smoothing-to-gradient'' lemma implying that the inclusions $\stab_0^r(\R^n)\hookrightarrow \stab_0^s(\R^n)$, $\stab^r(\R^n)\hookrightarrow \stab^s(\R^n)$ and  negative gradient embeddings $\fun_0^{r+1}(\R^n)\hookrightarrow \stab_0^r(\R^n)$, $\fun^{r+1}(\R^n) \hookrightarrow \stab^r(\R^n)$ are  weak homotopy equivalences for $r\geq s$.
These results finish the proofs of Theorems~\ref{th:intro-1-conn-contract}, \ref{th:intro-cr-extensions} for $\stab_0^r(\R^n)$, and the results of section~\ref{subsec:htpy-groups-complete-vf} finish the proofs for $\stabc_0^r(\R^n)$.
Section~\ref{subsec:htpy-groups-complete-vf} contains a simple rescaling lemma implying that the inclusions $\stabc_0^r(\R^n)\hookrightarrow \stab_0^r(\R^n)$, $\stabc^r(\R^n)\hookrightarrow \stab^r(\R^n)$ are weak homotopy equivalences.

\subsection{Smoothing-to-gradient lemma and homotopy groups}\label{subsec:smoothing-to-gradient-htpy-groups}

Approximation statements as in Lemma~\ref{lem:smoothing-funnels} could be added to the following lemma if its third conclusion is dropped, but the formulation below is adequate for our purposes.
The proof of the lemma relies on Lyapunov theory.

\begin{Lem}\label{lem:smoothing-vf}
Fix $r,u\in \N_{\geq 1}\cup\{\infty\}$ and $s,v,w\in \N_{\geq 0}\cup \{\infty\}$ satisfying $r\geq s \geq v$, $r\geq u \geq v$, and $s+1\geq w$. 
Let $\param$ be a compact $C^\infty$ manifold with boundary, $S\subset \param$ be a closed subset and $C^\infty$ neatly embedded submanifold of either $\param$ or $\partial \param$, and  $\vo\colon \param \to \stab^u(\R^n)$ be a continuous and $C^v$ map that restricts to a continuous and $C^s$ map $S\to \stab^r(\R^n)$ on $S$.
Then there is a  continuous and $C^v$ map $H\colon I\times \param \to \stab^u(\R^n)$ that is $\stab_0^u(\R^n)$-valued if $\vo$ is and satisfies
\begin{itemize}
	\item $H$ is a  homotopy $\rel{S}$ from $H_0 = \vo$ to a continuous and $C^s$ map $H_1\colon \param \to \stab^r(\R^n)$, and 
	\item  if there is a continuous and $C^w$ map $\lyap\colon S\to \fun^{r+1}(\R^n)$  such that $\vo(p) = -\nabla \lyap(p)$ for all $p\in S$,  then $H_1(p)=-\nabla \lyapt(p)$ for a continuous and $C^w$ map $\lyapt\colon \param\to \fun^{r+1}(\R^n)$ satisfying $\lyapt|_S = \lyap$.
\end{itemize}
\end{Lem}

%\iffalse
\begin{proof}
Using a closed tubular neighborhood $U\subset \param$ of $S$ \cite[p.~53]{kosinski1993differential},  a $C^\infty$ homotopy $\rho\colon I\times \param \to \param$ can be constructed so that  $\rho_0=\id_\param$, $\rho_1(U)=S$, and $\rho_t|_S = \id_S$ for all $t\in I$.
The continuous and $C^v$ map  $g\colon I\times \param \to \stab^u(\R^n)$ defined by $g_t=\vo\circ \rho_t$ is a homotopy $\rel{S}$ from $\vo$ to $\tilde{\vo}\coloneqq \vo\circ \rho_1$, and $g$ is $\stab_0^r(\R^n)$-valued if $\vo$ is.
Note that the restriction of $g_1$ to $U$ is a continuous and $C^s$ map $U\to \stab^r(\R^n)$.

For any compact neighborhood $L$ of $$\tilde{Z}\coloneqq \{(p,x)\in \param \times \R^n \colon \tilde{\vo}(p)(x)=0\},$$ any backward trajectory of the (uniquely integrable) vector field $(p,x)\mapsto \tilde{\vo}(p)(x)$ initialized in $L\setminus \tilde{Z}$ must eventually leave $L$ since $\tilde{\vo}$ is $\stab^u(\R^n)$-valued.
Thus, $\tilde{Z}$ is a (globally) asymptotically stable compact invariant set for the same vector field \cite[Lem.~3.1]{salamon1985connected}.
Wilson's converse Lyapunov function theorem then readily furnishes a $C^\infty$ map $\lyapth\colon \param\to \fun^\infty(\R^n)$ satisfying  $\lyapth(p)(x)=0$ for all $(p,x)\in \tilde{Z}$ and 
\begin{equation}\label{eq:inner-product-condition-vf-1}
	\ip{\nabla \lyapth(p)(x)}{\tilde{\vo}(p)(x)} < 0
\end{equation}
for all $(p,x)\in (\param\times \R^n)\setminus \tilde{Z}$ (\cite[Thm~3.2]{wilson1969smooth},  \cite[Sec.~6]{fathi2019smoothing}).
  
Let $\psi\colon \param\to I$ be a $C^\infty$ function equal to $0$ on $S$ and $1$ on a neighborhood of $P\setminus U$.
Define a continuous and $C^v$ homotopy  $h\colon I \times \param \to \vf^u(\R^n)$ by
\begin{equation}\label{eq:final-vf-smoothing-htpy}
	h_t(p)= -t\psi(p)\nabla \lyapth(p) + (1-t\psi(p))\tilde{\vo}(p).
\end{equation}
Note that $h_1$ is a continuous and $C^s$ map $h_1\colon \param \to \vf^r(\R^n)$ since $\lyapth$ is $C^\infty$ and $\psi$ is $1$ on a neighborhood of $\param \setminus U$.
Moreover, linearity and \eqref{eq:inner-product-condition-vf-1} imply that $$\ip{\nabla \lyapth(p)(x)}{h_t(p)(x)} < 0$$
for all $t\in I$ and $(p,x)\in (\param \times \R^n)\setminus \tilde{Z}$. 
Hence $\lyapth(p)$ is a proper $C^\infty$ Lyapunov function for $h_t(p)$ for each $(t,p)\in I\times \param$, so $h$ is in fact an $\stab^u(\R^n)$-valued map $h\colon I\times \param \to \stab^u(\R^n)$.
Thus, $h_1$ is a continuous and $C^s$ map $h_1\colon \param \to \stab^r(\R^n)$.
Note that $h$ is a homotopy $\rel{S}$ since $\psi$ is $0$ on $S$, and $h_t(p)(x)=0$ for all $t\in I$ and $(p,x)\in \tilde{Z}$ by linearity and the corresponding property of $\lyapth$ and of $\tilde{\vo}$.
And $h$ is $\stab_0^u(\R^n)$-valued if $\vo$ is, since then $\tilde{Z}=\param \times (\R^n\setminus \{0\})$.

Thus, smoothly concatenating (as in \eqref{eq:gamma-def}) the homotopies $g$ from $\vo$ to $\tilde{\vo}$ and $h$ from $\tilde{\vo}$ to $h_1$  yields a continuous and $C^v$ homotopy $H\colon I\times \param \to \stab^u(\R^n)$ having the desired  properties from $\vo$ to the continuous and $C^s$ map $H_1 = h_1\colon \param \to \stab^r(\R^n)$.

Finally, suppose that $\vo(p)=-\nabla \lyap(p)$ for all $p\in S$ and a continuous and $C^w$ map $\lyap\colon S\to \fun^{r+1}(\R^n)$.
Then $\tilde{F}(p)=\vo(\rho_1(p))=-\nabla \lyap (\rho_1(p))$ for all $p\in U$, so 
\begin{align*}
h_1(p)&=-\psi(p)\nabla \lyapth(p) - (1-\psi(p))\nabla \lyap(\rho_1(p))\\
&= -\nabla \Big(\psi(p)\lyapth(p)+(1-\psi(p))\lyap(\rho_1(p))\Big)
\end{align*}
for all $p\in \param$.
Defining the continuous and  $C^w$ map $\lyapt\colon \param \to \fun^{r+1}(\R^n)$ by $$\lyapt\coloneqq \Big(\psi \lyapth + (1-\psi)\lyap\circ \rho_1 \Big)$$ finishes the proof.
\end{proof}
%\fi

\begin{Co}\label{co:whe-vfs-funnels}
	Fix $r,s\in \N_{\geq 1}\cup\{\infty\}$ with $r\geq s$.
	The inclusion maps $$\stab_0^r(\R^n)\hookrightarrow \stab_0^s(\R^n), \quad \stab^r(\R^n)\hookrightarrow \stab^s(\R^n)$$ and negative gradient embeddings $$\fun_0^{r+1}(\R^n)\hookrightarrow \stab_0^{r}(\R^n), \quad \fun^{r+1}(\R^n)\hookrightarrow \stab^{r}(\R^n)$$ defined by $\lyap \mapsto (p\mapsto -\nabla \lyap(p))$ are all weak homotopy equivalences.
\end{Co}

%\iffalse
\begin{proof}
	For each $k\in \N_{\geq 0}$,  all homomorphisms of $k$-th homotopy groups induced by the stated maps are surjective by Lemma~\ref{lem:smoothing-vf} with $(\param,S)=(\sph^k, \{*\})$ and $(r,s,u,v)$ replaced by $(r,0,s,0)$, and injective by Lemma~\ref{lem:smoothing-vf} with $(\param, S)=(D^{k+1}, \sph^k)$ and $(r,s,u,v)$ replaced by $(r,0,s,0)$.
\end{proof}
%\fi

The following pair of corollaries are two of our main results.
They complete the proof of Theorems~\ref{th:intro-1-conn-contract}, \ref{th:intro-cr-extensions}, \ref{th:intro-weak-htpy-vfs-etc} for $\stab_0^r(\R^n)$.

\begin{Co}\label{co:vf-homotopy-groups}
	For any $r\in \N_{\geq 1}\cup \{\infty\}$ and $n\neq 4,5$, $\stab_0^r(\R^n)$ has the weak homotopy type of $\U(D^n,\R^n)$. In particular, for each $k\in \N_{\geq 0}$ there is an isomorphism $$\pi_k \stab_0^r(\R^n)\cong \pi_k \U(D^n,\R^n).$$ 	
\end{Co}

%\iffalse
\begin{proof}
This is immediate from Theorem~\ref{th:funnel-space-fiber-bundle-over-U}, Lemma~\ref{lem:gr-0-gr-w-h-e}, and Corollary~\ref{co:whe-vfs-funnels}.
\end{proof}
%\fi

\begin{Co}\label{co:vf-trivial-homotopy-groups}
	For any $r\in \N_{\geq 1} \cup \{\infty\}$,  $\stab_0^r(\R^n)$  is path-connected and simply connected if $n\neq 4,5$ and weakly contractible if $n\leq 3$.
	In fact, if $\param$ is a compact $C^\infty$ manifold with boundary, $S\subset \param$ is a closed subset and $C^\infty$ neatly embedded submanifold of either $\param$ or $\partial \param$, and $\vo_S\colon S \to \stab_0^r(\R^n)$ is continuous and $C^s$ with $0 \leq s\leq r$, then there is a continuous and $C^s$ map $\vo\colon \param \to \stab_0^r(\R^n)$ extending $\vo_S$ if either (i) $n\neq 4,5$ and $\dim \param \leq 2$ or (ii) $n\leq 3$.
\end{Co}

%\iffalse
\begin{proof}
Corollaries~\ref{co:funnels-trivial-homotopy-groups}, \ref{co:whe-vfs-funnels} directly imply the first statement.
The existence of a continuous map $\tilde{\vo}\colon \param \to \stab_0^r(\R^n)$ equal to $\vo_S$ on $S$ then follows from the first statement and obstruction theory \cite[Cor.~VII.13.13]{bredon1993topology}.
Lemma~\ref{lem:smoothing-vf} with $(r,s,u,v)$ replaced by $(r,s,r,0)$  then furnishes a continuous and $C^s$ map $\vo\colon \param \to \stab_0^r(\R^n)$ equal to $\vo_S$ on $S$.	
\end{proof}
%\fi

\subsection{Homotopy groups for complete vector fields}\label{subsec:htpy-groups-complete-vf}

The following lemma is based on the simple observation that any vector field can be made complete after multiplying it by a suitable positive function.
\begin{Lem}\label{lem:completing-vfs}
	Fix  $r\in \N_{\geq 1}\cup \{\infty\}$ and $s\in \N_{\geq 0}\cup \{\infty\}$ with $s\leq r$.
	Let $\param$ be a $C^\infty$ manifold with corners, $S\subset \param$ be a closed subset, and  $\vo\colon \param\to \vf^r(\R^n)$ be a continuous and $C^s$ map.
    Then there is a continuous and $C^s$ map  $H\colon I\times \param \to \vf^r(\R^n)$ such that
	\begin{itemize}
	\item $H$ is a  homotopy $\rel{S}$ from $H_0 = \vo$ to a continuous and  $C^s$ map $H_1\colon \param \to \vf^r(\R^n)$ such that $H_1|_{\param \setminus S}$ is $\vfc^r(\R^n)$-valued, 
	\item $\{(p,x)\colon H_t(p)(x)=0\}=\{(p,x)\colon \vo(p)(x)=0\}$ for all $t\in I$, and
	\item if $\vo|_{\param \setminus S}$ is $\stab^r(\R^n)$-valued,  then $H|_{I\times (\param \setminus S)}$ is $\stab^r(\R^n)$-valued and hence $H_1|_{\param \setminus S}$ is $\stabc^r(\R^n)$-valued.
	\end{itemize}
\end{Lem}

%\iffalse
\begin{proof}
Let $\varphi\colon \param \to [0,\infty)$ be $C^\infty$ and satisfy $\varphi^{-1}(0)=S$.
Define the continuous and $C^s$ map $H\colon I\times \param\to \vf^r(\R^n)$ by 
\begin{equation*}
	H_t(p)(x) = \frac{1}{1+t\varphi(p)\norm{\vo(p)(x)}^2}\vo(p)(x)
\end{equation*}
and note that $H$ is a homotopy $\rel{S}$ satisfying $$\{(p,x)\colon H_t(p)(x)=0\}=\{(p,x)\colon \vo(p)(x)=0\}$$ for all $t\in I$.
Moreover, for each fixed $p\in \param\setminus S$ there is $K_p > 0$ such that $\norm{H_1(p)(x)}<K_p$ for all $x\in \R^n$.
Hence $H_1(p)$ is complete for each $p\in \param\setminus  S$.
If $\vo|_{\param \setminus S}$ is $\stab^r(\R^n)$-valued, then so is $H|_{I\times (\param \setminus S)}$ since $H_t(p)$ is the product of $\vo(p)$ with a positive function for each $t\in I$, $p\in \param$.   
\end{proof}
%\fi

\begin{Co}\label{co:whe-vfs-complete}
	For any $r\in \N_{\geq 1}\cup\{\infty\}$ the inclusion maps $\stabc_0^r(\R^n)\hookrightarrow \stab_0^r(\R^n)$ and $\stabc^r(\R^n)\hookrightarrow \stab^r(\R^n)$ are weak homotopy equivalences.
\end{Co}

%\iffalse
\begin{proof}
	For each $k\in \N_{\geq 0}$,  both homomorphisms of $k$-th homotopy groups induced by the stated maps are surjective by Lemma~\ref{lem:completing-vfs} with $(\param,S)=(\sph^k, \{*\})$ and $s=0$, and injective by Lemma~\ref{lem:completing-vfs} with $(\param, S)=(D^{k+1}, \sph^k)$ and $s=0$.
\end{proof}
%\fi

The following pair of corollaries complete the proof of Theorems~\ref{th:intro-1-conn-contract}, \ref{th:intro-cr-extensions}, \ref{th:intro-weak-htpy-vfs-etc} for $\stabc_0^r(\R^n)$.

\begin{Co}\label{co:complete-vf-homotopy-groups}
	For any $r\in \N_{\geq 1}\cup \{\infty\}$ and $n\neq 4,5$, $\stabc_0^r(\R^n)$ has the weak homotopy type of $\U(D^n,\R^n)$. 
	In particular, for each $k\in \N_{\geq 0}$ there is an isomorphism $$\pi_k \stabc_0^r(\R^n)\cong \pi_k \U(D^n,\R^n).$$ 	
\end{Co}

%\iffalse
\begin{proof}
	This is immediate from Corollaries~\ref{co:vf-homotopy-groups}, \ref{co:whe-vfs-complete}.
\end{proof}
%\fi

\begin{Co}\label{co:complete-vf-trivial-homotopy-groups}
	For any $r\in \N_{\geq 1} \cup \{\infty\}$,  $\stabc_0^r(\R^n)$  is path-connected and simply connected if $n\neq 4,5$ and weakly contractible if $n\leq 3$.
	In fact, if $\param$ is a compact $C^\infty$ manifold with boundary, $S\subset \param$ is a closed subset and $C^\infty$ neatly embedded submanifold of either $\param$ or $\partial \param$, and $\vo_S\colon S \to \stabc_0^r(\R^n)$ is continuous and $C^s$ with $0 \leq s\leq r$, then there is a continuous and $C^s$ map $\vo\colon \param \to \stabc_0^r(\R^n)$ extending $\vo_S$ %such that $\vo|_{\param \setminus S}$ is $\stabc_0^r(\R^n)$-valued 
	if either (i) $n\neq 4,5$ and $\dim \param \leq 2$ or (ii) $n\leq 3$.
\end{Co}

%\iffalse
\begin{proof}
     Corollaries~\ref{co:vf-trivial-homotopy-groups}, \ref{co:whe-vfs-complete} directly imply the first statement.	
	The existence of a continuous and $C^s$ map $\tilde{\vo}\colon \param \to \stabc_0^r(\R^n)$ equal to $\vo_S$ on $S$ then follows from Corollary~\ref{co:vf-trivial-homotopy-groups} and Lemma~\ref{lem:completing-vfs}.
	%Lemma~\ref{lem:completing-vfs} then furnishes a continuous and $C^s$ map $\vo\colon \param \to \stab_0^r(\R^n)$ equal to $\vo_S$ on $S$ that is $\stabc_0^r(\R^n)$-valued on $\param \setminus S$.
\end{proof}
%\fi

\section{Path-connectedness and the $4$-dimensional smooth Poincar\'{e} conjecture}\label{sec:path-conn-poincare}

As mentioned in the introduction, we show that validity of the path-connected portion of Theorem~\ref{th:intro-1-conn-contract} for $n=5$ would imply that the $4$-dimensional smooth Poincar\'{e} conjecture is true.
This conjecture asserts that every $4$-dimensional  $C^\infty$ manifold homotopy equivalent to the $4$-sphere is diffeomorphic to the $4$-sphere.  

\begin{Prop}\label{prop:path-conn-poincare}
The $4$-dimensional smooth Poincar\'{e} conjecture is true if any of $\fun_0^r(\R^5)$, $\stab_0^r(\R^5)$, or $\stabc_0^r(\R^5)$ are path-connected for some $r\in \N_{\geq 1}\cup \{\infty\}$.
\end{Prop}

\begin{Rem}\label{rem:poincare}
	By Lemma~\ref{lem:varying-minima-equilibria} of section~\ref{sec:varying-min-eq}, Proposition~\ref{prop:path-conn-poincare} also holds for the spaces $\fun^r(\R^5)$, $\stab^r(\R^5)$, $\stabc^r(\R^5)$ with unconstrained locations of minima and equilibria.
\end{Rem}

%\iffalse
\begin{proof}
By Corollaries~\ref{co:whe-funnels}, \ref{co:whe-vfs-funnels}, \ref{co:whe-vfs-complete} it suffices to show that path-connectedness of $\fun_0^\infty(\R^5)$ implies that the $4$-dimensional smooth Poincar\'{e} conjecture is true.
So, assume $\fun_0^\infty(\R^5)$ is path-connected and let $M$ be a $4$-dimensional $C^\infty$ manifold homotopy equivalent to $\sph^4$.
By a theorem of Hirsch, there exists $\lyap\in \fun_0^\infty(\R^n)$ such that $M$ is diffeomorphic to $\lyap^{-1}(1)\subset \R^5$ \cite[Thm~2]{hirsch1965homotopy}.
Define $\lyapt\in \fun_0^\infty(\R^n)$ by $\lyapt(x)=\norm{x}^2$.
The path-connectedness assumption and Lemma~\ref{lem:smoothing-funnels} furnish a $C^\infty$ map $J\colon I\to \fun_0^\infty(\R^n)$ satisfying $J(0)=\lyap$ and $J(1)=\lyapt$.
The domain of the map
\begin{equation}\label{eq:prop:path-conn-poincare}
\pr_1 \colon \{(t,x)\in I\times \R^n \colon J(t)(x)=1\}	\subset I\times \R^5\to I
\end{equation}
is a $C^\infty$ manifold transverse to all slices $\{t\}\times \R^n$ by the implicit function theorem.
Thus, \eqref{eq:prop:path-conn-poincare} is a submersion and hence also a $C^\infty$ fiber bundle by Ehresmann's lemma.
In particular, the fibers $\pr_1^{-1}(0)\approx M$ and $\pr_1^{-1}(1)=\sph^4$ are diffeomorphic.
\end{proof}
%\fi

\section{Almost nonsingular vector fields and relative homotopy groups}\label{sec:relative-htpy-groups}

In this section, we prove that the inclusions $\stab_0^r(\R^n)\hookrightarrow \AN_0^r(\R^n)$, $\stabc_0^r(\R^n)\hookrightarrow \ANc_0^r(\R^n)$ are nullhomotopic and derive consequences for relative homotopy groups.

\begin{Th}\label{th:nullhomotopic-inclusion-s-an}
For any $r\in \N_{\geq 1}\cup \infty$, the inclusion $\stab_0^r(\R^n) \hookrightarrow\AN_0^r(\R^n)$ is nullhomotopic.
Moreover, if $\param$ is a $C^\infty$ manifold with boundary and $\vo\colon \param \to \stab_0^r(\R^n)$ is a continuous and $C^s$ map with $s\leq r$, then there is a continuous and $C^s$ homotopy $I\times \param \to \AN_0^r(\R^n)$ from $\vo$ to the constant map $\param\to  \{x\mapsto -x\}\in \AN_0^r(\R^n)$.
\end{Th}

\begin{Rem}
	A similar proof of the weaker statement that any $\vo\in \stab_0^r(\R^n)$ is connected by a path in $\AN_{0}^r(\R^n)$ to $x\mapsto -x$ previously appeared in  \cite[p.~340]{krasnoselskii1984geometrical}, \cite[p.~255]{sontag1998mathematical}, and the first homotopy formula \eqref{eq:vf-htpy} %, \eqref{eq:vf-htpy-Cs}
	 is essentially what appears in the latter reference.
	Here we observe that the same formula works parametrically to imply the stronger conclusion here. 
\end{Rem}

%\iffalse
\begin{proof}
Observe that the continuous map $H\colon I\times \stab_0^r(\R^n)\to \AN_0^r(\R^n)$ defined by
\begin{equation}\label{eq:vf-htpy}
H_t(\vo)(x)=\begin{cases}
	\vo(x), & t=0\\
	-x, &t=1\\
	\frac{1}{t}\left(\Phi(\vo)^{\frac{t}{1-t}}(x)-x\right), & 0<t<1
\end{cases} 
\end{equation}
is a homotopy from the inclusion $\stab_0^r(\R^n)\hookrightarrow \AN_0^r(\R^n)$ to the constant map sending all of $\stab_0^r(\R^n)$ to the linear vector field $x\mapsto -x$, where $\Phi(\vo)$ is the $C^r$ (semi)flow of $\vo\in \stab_0^r(\R^n)$.
This is the desired nullhomotopy. 

Next, let $\param$ be a $C^\infty$ manifold with boundary and $\vo\colon \param \to \stab_0^r(\R^n)$ be a continuous and $C^s$ map. 
Let $\Phi(p)$ be the semiflow generated by $p\in  \param$.
By a standard fact from ordinary differential equations,  $\Phi\colon \param \to C^r([0,\infty)\times \R^n, \R^n)$ is continuous and $C ^s$ \cite[Thm~B.3]{duistermaat2000lie}.

Continuity of $(p,t,x)\mapsto \Phi(p)^t(x)$ and the definition of asymptotic stability imply that, for each $p\in \param$, $x\in \R^n$ there is $T_{p,x}>0$ and an open neighborhood $U_{p,x}\subset \param \times \R^n$ of $(p,x)$ such that $\norm{\Phi(q)^t(y)}\leq \frac{1}{2}\norm{y}$ for all $(q,y)\in U_{p,x}$ and $t\geq T_{p,x}$.
Let $(\psi_{p,x})_{(p,x)\in \param \times \R^n}$ be a $C^\infty$ partition of unity subordinate to the open cover $(U_{p,x})_{(p,x)\in \param \times \R^n}$ of $\param \times \R^n$, and define the $C^\infty$ function $\tau\colon \param \times \R^n\to (0,\infty)$ by  
\begin{equation*}
	\tau(p,x)=\sum_{(q,y)\in \param\times \R^n}\psi_{q,y}(p,x)T_{q,y}.	
\end{equation*}
(The sum is well-defined because the family $(\supp \psi_{p,x})_{(p,x)\in \param \times \R^n}$ of supports is locally finite by definition \cite[p.~43]{lee2013smooth}.)
Then for each $(p,x)\in \param \times \R^n$ there is $(q,y)\in \param \times \R^n$ such that $(p,x)\in U_{q,y}$ and $\tau(p,x)\geq T_{q,y}$.
Hence for all $(p,x)\in \param \times \R^n$,
\begin{equation}\label{eq:norm-halve}
	\norm{\Phi(p)^{\tau(p,x)}(x)}\leq \frac{1}{2}\norm{x}.
\end{equation}

Since $\vo$ and $\Phi$ are both continuous and $C^s$ and  
\begin{align*}
	\Phi(p)^t(x)-x&=\int_0^t \vo(p)(\Phi(p)^s(x))\, ds\\
	&= \underbrace{\left(\int_0^1 \vo(p)(\Phi(p)^{tu}(x))\, du\right)}_{g(p,t,x)} t,
\end{align*}
there is a continuous and $C^s$ map $g\colon \param \times I \times \R^n\to \AN_0^r(\R^n)$ such that $\Phi(p)^t(x)-x=g(p,t,x)t$ and $g(p,0,x)=\vo(p)(x)$.
Thus, the formula
\begin{equation}\label{eq:quasi-sontag}
	(t,p)\mapsto \begin{cases}
		\frac{1}{t \tau(p,x)}\left(\Phi(p)^{t\tau(p,x)}(x)-x\right), & 0 < t \leq 1\\
		\vo(p)(x), &t=0
	\end{cases}
\end{equation}
defines a continuous and $C^s$ homotopy from $\vo$ to the map
\begin{equation}\label{eq:phi-tau-vf}
	p\mapsto \left(x\mapsto  \frac{1}{\tau(p,x)}\left(\Phi(p)^{\tau(p,x)}(x)-x\right)\right).
\end{equation}

Finally, \eqref{eq:norm-halve} implies that the straight-line homotopy from \eqref{eq:phi-tau-vf} to the constant map  $p\mapsto (x\mapsto -x)$ is $\AN_0^r(\R^n)$-valued, so smoothly concatenating (as in \eqref{eq:gamma-def}) this homotopy with  \eqref{eq:quasi-sontag} yields the desired continuous and $C^s$ homotopy $I\times \param \to \AN_0^r(\R^n)$ from $\vo$ to the contant map $p\mapsto (x\mapsto -x)$.
\end{proof}
%\fi

\begin{Co}\label{co:rel-htpy-grp-dir-sum}
	For any $k\in \N_{\geq 0}$ and $r\in \N_{\geq 1}\cup \{\infty\}$, there are isomorphisms
\begin{align*}%\label{eq:co:rel-htpy-grp-dir-sum}
	\pi_{k+1}(\AN_0^r(\R^n), \stab_0^r(\R^n)) &\cong \pi_{k+1}\AN_0^r(\R^n) \oplus \pi_k \stab_0^r(\R^n)\\
	&\cong \pi_{k+1} \AN_0^r(\R^n) \oplus \pi_k\U(D^n,\R^n)
\end{align*}
\end{Co}

%\iffalse
\begin{proof}
	The first isomorphism follows from Theorem~\ref{th:nullhomotopic-inclusion-s-an} and the long exact sequence of homotopy groups of the pair $(\AN_0^r(\R^n), \stab_0^r(\R^n))$ \cite[p.~448]{bredon1993topology}.
	The second isomorphism follows from Corollary~\ref{co:vf-homotopy-groups}.
\end{proof}
%\fi

\begin{Co}\label{co:rel-htpy-grp-dir-sum-collapsed-s-c-h}
	Fix $k\in \N_{\geq 0}$ and $r\in \N_{\geq 1}\cup \{\infty\}$.
	Assume that either (i) $n\geq 6$ and $k=0,1$ or (ii) $n\leq 3$.
	Then there is an isomorphism
	\begin{equation}\label{eq:co:rel-htpy-grp-dir-sum-collapsed-s-c-h}
		\pi_{k+1}(\AN_0^r(\R^n), \stab_0^r(\R^n)) \cong \pi_{k+1} \AN_0^r(\R^n).
	\end{equation}
\end{Co}

%\iffalse
\begin{proof}
This is immediate from Corollaries~\ref{co:vf-trivial-homotopy-groups}, \ref{co:rel-htpy-grp-dir-sum}.	
\end{proof}
%\fi

The following result is a version of Theorem~\ref{th:nullhomotopic-inclusion-s-an} for complete vector fields.

\begin{Th}\label{th:nullhomotopic-inclusion-complete-s-an}
	For any $r\in \N_{\geq 1}\cup \infty$, the inclusion $\stabc_0^r(\R^n) \hookrightarrow \ANc_0^r(\R^n)$ is nullhomotopic.
	Moreover, if $\param$ is a $C^\infty$ manifold with boundary and $\vo\colon \param \to \stabc_0^r(\R^n)$ is a continuous and $C^s$ map with $s\leq r$, then there is a continuous and $C^s$ homotopy $I\times \param \to \ANc_0^r(\R^n)$ from $\vo$ to the constant map $\param\to  \{x\mapsto -x\}\in \ANc_0^r(\R^n)$.
\end{Th}

%\iffalse
\begin{proof}
Let $H\colon I \times \stabc_0^r(\R^n) \to \AN_0^r(\R^n)$ be the indicated restriction of the nullhomotopy of Theorem~\ref{th:nullhomotopic-inclusion-s-an}.
Let $\varphi\colon I\to [0,\infty)$ be a $C^\infty$ function satisfying $\varphi^{-1}(0)=\{0\}$.
Define the continuous and $C^s$ nullhomotopy $\tilde{H}\colon I \times \stabc_0^r(\R^n) \to \AN_0^r(\R^n)$ by 
\begin{equation*}
	\tilde{H}_t(\vo)(x)=\frac{1}{1+\varphi(t)\norm{H_t(\vo)(x)}^2}H_t(\vo)(x).
\end{equation*}
Note that $\tilde{H}_0$ is the inclusion $\stabc_0^r(\R^n)\hookrightarrow \ANc_0^r(\R^n)$ and that, for each fixed $t \in (0,1]$ and $\vo\in \stabc_0^r(\R^n)$, there is $K_{t,\vo} > 0$ such that $\norm{\tilde{H}_t(\vo)(x)} < K_{t,\vo}$ for all $x\in \R^n$.
Hence $\tilde{H}_t(\vo)$ is complete for each $t\in (0,1]$ and $\vo\in \stabc_0^r(\R^n)$, so $\tilde{H}$ is the desired $\ANc_0^r(\R^n)$-valued nullhomotopy.

Next, let $\param$ be a $C^\infty$ manifold with boundary and $\vo\colon \param \to \stabc_0^r(\R^n)$ be a continuous and $C^s$ map.
Theorem~\ref{th:nullhomotopic-inclusion-s-an} followed by Lemma~\ref{lem:completing-vfs} (with $(\param, S)$ replaced by $(I\times \param, \{0,1\} \times \param)$) furnishes the desired continuous and $C^s$ homotopy $H\colon I\times \param \to \ANc_0^r(\R^n)$ from $\vo$ to the constant map $\param \to \{x\mapsto -x\}\in \ANc_0^r(\R^n)$.
\end{proof}
%\fi

\begin{Co}\label{co:complete-rel-htpy-grp-dir-sum}
	For any $k\in \N_{\geq 0}$ and $r\in \N_{\geq 1}\cup \{\infty\}$, there are isomorphisms
	\begin{align*}%\label{eq:co:rel-htpy-grp-dir-sum}
		\pi_{k+1}(\ANc_0^r(\R^n), \stabc_0^r(\R^n)) &\cong \pi_{k+1} \ANc_0^r(\R^n) \oplus \pi_k \stabc_0^r(\R^n)\\
		&\cong \pi_{k+1} \ANc_0^r(\R^n) \oplus \pi_k\U(D^n,\R^n)
	\end{align*}
\end{Co}

%\iffalse
\begin{proof}
	The first isomorphism follows from Theorem~\ref{th:nullhomotopic-inclusion-complete-s-an} and the long exact sequence of homotopy groups of the pair $(\ANc_0^r(\R^n), \stabc_0^r(\R^n))$ \cite[p.~448]{bredon1993topology}.
	The second isomorphism follows from Corollary~\ref{co:complete-vf-homotopy-groups}.
\end{proof}
%\fi

\begin{Co}\label{co:complete-rel-htpy-grp-dir-sum-collapsed-s-c-h}
	Fix $k\in \N_{\geq 0}$ and $r\in \N_{\geq 1}\cup \{\infty\}$.
	Assume that either (i) $n\geq 6$ and $k=0,1$ or (ii) $n\leq 3$.
	Then there is an isomorphism
	\begin{equation}\label{eq:co:rel-htpy-grp-dir-sum-collapsed-s-c-h}
		\pi_{k+1}(\ANc_0^r(\R^n), \stabc_0^r(\R^n)) \cong \pi_{k+1} \ANc_0^r(\R^n).
	\end{equation}
\end{Co}

%\iffalse
\begin{proof}
	This is immediate from Corollaries~\ref{co:complete-vf-trivial-homotopy-groups}, \ref{co:complete-rel-htpy-grp-dir-sum}.
\end{proof}
%\fi

\section{Varying minima and equilibria}\label{sec:varying-min-eq}

For the spaces $\fun_0^r(\R^n)$, $\stab_0^r(\R^n)$, $\stabc_0^r(\R^n)$ considered in the preceding results, the unique minima and equilibria were constrained to be at the origin.
In this section, we extend those results to allow the minima and equilibria to have unconstrained locations.
This is done using the following basic lemma.

\begin{Lem}\label{lem:varying-minima-equilibria}
Fix $r\in \N_{\geq 1}\cup \{\infty\}$ and $y\in \R^n$.
Then $\fun_y^r(\R^n)$ is a strong deformation retract of $\fun^r(\R^n)$, and there is a strong deformation retraction of $\AN^r(\R^n)$ onto $\AN_y^r(\R^n)$ that also restricts to strong deformation retractions of (i)  $\ANc^r(\R^n)$ onto $\ANc_y^r(\R^n)$, (ii) $\stab^r(\R^n)$ onto $\stab_y^r(\R^n)$, and (iii) $\stabc^r(\R^n)$ onto $\stabc_y^r(\R^n)$.
\end{Lem}

%\iffalse
\begin{proof}
Let $\F$ be either  $\fun^r(\R^n)$ or $\AN^r(\R^n)$, and let $\F_y\subset \F$ denote those maps having $y\in \R^n$ as their unique minimum or equilibrium.

Let $x_*\colon \F \to \R^n$ be the map sending each $f\in \F$ to its unique minimum or equilibrium $x_*(f)$, which is readily seen to be continuous with respect to the compact-open $C^r$ topology on $\F$.
Define a continuous map $H\colon I\times \F \to \F$ by $$H_t(f)(x)=f\left(x+(x_*(f)-y)t \right).$$
This is a strong deformation retraction of $\F$ onto $\F_y$.
And if $\F=\AN^r(\R^n)$, then $H_t$ clearly preserves each of $\ANc^r(\R^n)$, $\stab^r(\R^n)$, and $\stabc^r(\R^n)$ for each $t\in I$, so $H$ also restricts to the desired strong deformation retractions.
\end{proof}
%\fi

\begin{Co}\label{co:varying-minima-funnel-whe}
		For $n\neq 4,5$ there is a weak homotopy equivalence  $\fun^\infty(\R^n)\to \U(D^n,\R^n)$.
\end{Co}

%\iffalse
\begin{proof}
By Lemma~\ref{lem:varying-minima-equilibria} there is a homotopy equivalence $\fun^\infty(\R^n)\to \fun_0^\infty(\R^n)$.
The composition of this map followed by the weak homotopy equivalences $\fun_0^\infty(\R^n)\to \U_0(D^n,\R^n)$ and $\U_0(D^n,\R^n) \hookrightarrow \U(D^n,\R^n)$ of Theorem~\ref{th:funnel-space-fiber-bundle-over-U} and Lemma~\ref{lem:gr-0-gr-w-h-e}, respectively,
is the desired equivalence.
\end{proof}
%\fi

The following pair of corollaries complete the proof of Theorems~\ref{th:intro-1-conn-contract}, \ref{th:intro-cr-extensions}, \ref{th:intro-weak-htpy-vfs-etc} for $\fun^r(\R^n)$, $\stab^r(\R^n)$, and $\stabc^r(\R^n)$.

\begin{Co}\label{co:varying-equilibria-minima-htpy-grps}
	For any $r\in \N_{\geq 1}\cup \{\infty\}$ and $n\neq 4,5$,  $\fun^r(\R^n)$, $\stab^r(\R^n)$, and $\stabc^r(\R^n)$ have the weak homotopy type of $\U(D^n,\R^n)$. 
	In particular, for each $k\in \N_{\geq 0}$ there are isomorphisms 
$$\pi_k \fun^r(\R^n)\cong \pi_k \stab^r(\R^n) \cong \pi_k \stabc^r(\R^n) \cong \pi_k \U(D^n,\R^n).$$ 	
\end{Co}

%\iffalse
\begin{proof}
	This is immediate from Lemma~\ref{lem:varying-minima-equilibria} and Corollaries~\ref{co:funnels-homotopy-groups}, \ref{co:vf-homotopy-groups}, \ref{co:complete-vf-homotopy-groups}.
\end{proof}
%\fi

\begin{Co}\label{co:varying-eq-min-trivial-homotopy-groups}
Fix $r\in \N_{\geq 1}\cup \{\infty\}$ and let $\F$ denote any of $\fun^r(\R^n)$, $\stab^r(\R^n)$, or $\stabc^r(\R^n)$. 
Then $\F$  is path-connected and simply connected if $n\neq 4,5$ and weakly contractible if $n\leq 3$.
In fact, if $\param$ is a compact $C^\infty$ manifold with boundary, $S\subset \param$ is a closed subset and $C^\infty$ neatly embedded submanifold of either $\param$ or $\partial \param$, and $f_S\colon S \to \F$ is continuous and $C^s$ with $0 \leq s\leq r$, then there is a continuous and $C^s$ map $f\colon \param \to \F$ extending $f_S$ if either (i) $n\neq 4,5$ and $\dim \param \leq 2$ or (ii) $n\leq 3$.
%Moreover, if $\F = \stab^r(\R^n)$ then we may further arrange that $f|_{\param \setminus S}$ is $\stabc^r(\R^n)$-valued.
\end{Co}

%\iffalse
\begin{proof}
Proposition~\ref{prop:nonlinear-grass-trivial-htpy-groups} and Corollary~\ref{co:varying-equilibria-minima-htpy-grps} directly imply the first statement.
The existence of a continuous map $\tilde{f}\colon \param \to \F$ equal to $f_S$ on $S$ then follows from the first statement and obstruction theory \cite[Cor.~VII.13.13]{bredon1993topology}.
Lemmas~\ref{lem:smoothing-funnels}, \ref{lem:smoothing-vf}, \ref{lem:completing-vfs}
 then furnish a continuous and $C^s$ map $f\colon \param \to \F$ equal to $f_S$ on $S$. % that is $\stabc^r(\R^n)$-valued on $\param \setminus S$ if $\F=\stab^r(\R^n)$. 	
\end{proof}
%\fi

The following pair of corollaries extend the ones from section~\ref{sec:relative-htpy-groups}.

\begin{Co}\label{co:varying-minima-equilibria-rel-htpy-grp-dir-sum}
	For any $r\in \N_{\geq 1}\cup \{\infty\}$ and $k\in \N_{\geq 0}$, there are isomorphisms
	\begin{align*}%\label{eq:co:varying-minima-equilibria-rel-htpy-grp-dir-sum}
		\pi_{k+1}(\AN^r(\R^n), \stab^r(\R^n)) &\cong \pi_{k+1} \AN^r(\R^n) \oplus \pi_k \stab^r(\R^n)\\
		&\cong \pi_{k+1} \AN^r(\R^n) \oplus \pi_k \U(D^n,\R^n)
	\end{align*}
and 	
	\begin{align*}%\label{eq:co:complete-varying-minima-equilibria-rel-htpy-grp-dir-sum}
	\pi_{k+1}(\ANc^r(\R^n), \stabc^r(\R^n)) &\cong \pi_{k+1}\ANc^r(\R^n) \oplus \pi_k \stabc^r(\R^n)\\
	&\cong \pi_{k+1}\ANc^r(\R^n) \oplus \pi_k \U(D^n,\R^n).
\end{align*}
\end{Co}

%\iffalse
\begin{proof}
This follows from Corollaries~\ref{co:rel-htpy-grp-dir-sum}, \ref{co:complete-rel-htpy-grp-dir-sum}, \ref{co:varying-equilibria-minima-htpy-grps} and the fact that Lemma~\ref{lem:varying-minima-equilibria} supplies deformation retractions of the pairs $(\AN^r(\R^n), \stab^r(\R^n))$ and $(\ANc^r(\R^n), \stabc^r(\R^n))$ onto $(\AN^r_0(\R^n), \stab^r_0(\R^n))$ and $(\ANc^r_0(\R^n), \stabc^r_0(\R^n))$, respectively, which induce isomorphisms \cite[p.~344]{hatcher2001algebraic} $$\pi_{k+1}(\AN^r(\R^n), \stab^r(\R^n))\cong \pi_{k+1}(\AN_0^r(\R^n), \stab_0^r(\R^n))$$ and $$\pi_{k+1}(\ANc^r(\R^n), \stabc^r(\R^n))\cong \pi_{k+1}(\ANc_0^r(\R^n), \stabc_0^r(\R^n)).$$ 
\end{proof}
%\fi

\begin{Co}\label{co:varying-minima-equilibria-rel-htpy-grp-dir-sum-collapsed-s-c-h}
	Fix $k\in \N_{\geq 0}$ and $r\in \N_{\geq 1}\cup \{\infty\}$.
Assume that either (i) $n\geq 6$ and $k=0,1$ or (ii) $n\leq 3$.
Then there are isomorphisms
	\begin{equation*}%\label{eq:co:varying-minima-equilibria-rel-htpy-grp-dir-sum-collapsed-s-c-h}
		\pi_{k+1}(\AN^r(\R^n), \stab^r(\R^n)) \cong \pi_{k+1}\AN^r(\R^n)
	\end{equation*}
and	
	\begin{equation*}%\label{eq:co:complete-varying-minima-equilibria-rel-htpy-grp-dir-sum-collapsed-s-c-h}
	\pi_{k+1}(\ANc^r(\R^n), \stabc^r(\R^n)) \cong \pi_{k+1}\ANc^r(\R^n).
\end{equation*}
\end{Co}

%\iffalse
\begin{proof}
This is immediate from Corollaries~\ref{co:varying-eq-min-trivial-homotopy-groups}, \ref{co:varying-minima-equilibria-rel-htpy-grp-dir-sum}.
\end{proof}
%\fi

\section{Applications}\label{sec:applications}

\subsection{A question of Conley}\label{subsec:conley}

In his monograph \cite{conley1978isolated}, Conley defined what is now called the Conley index \cite[Def.~III.5.2]{conley1978isolated} and showed that a pair $A$, $B$ of compact isolated invariant  sets  for continuous flows $\Phi$, $\Psi$ on a metrizable space $M$ have isomorphic  Conley indices if they are related by \emph{continuation} \cite[p.~65]{conley1978isolated}.
In particular, this is the case if there is a continuous family $(\Theta_s)_{s\in I}$ of flows $\Theta_s$ such that $\Theta_0=\Phi$, $\Theta_1=\Psi$, and the flow $(t,s,x)\mapsto (s,\Theta_s^t(x))$ on $I\times M$ has an isolated invariant set $C\subset I\times M$ with $C\cap (\{0\}\times M) = A$ and $C\cap (\{1\}\times M) = B$ (cf. \cite[p.~186]{reineck1992continuation}).

Conley asked to what extent the converse is true \cite[Sec,~IV.8.1.A]{conley1978isolated}.
That is, given a pair of compact isolated invariant sets with isomorphic Conley indices, when are they related by continuation?
The converse is not completely true, as known counterexamples show \cite[pp.~1--2]{mrozek2000conley}, but one can ask whether a converse holds under additional assumptions.

Reineck proved a fairly general partial converse for smooth flows on manifolds \cite{reineck1992continuation}.
Using a Lyapunov function theorem of Robbin and Salamon \cite{robbin1988dynamical}, Reineck first continues an arbitrary (compact) isolated invariant set to one of a Morse-Smale gradient flow \cite{reineck1991continuation}.
Under additional topological assumptions, he then uses critical point cancellation techniques from  the Morse-theoretic proof \cite{milnor1965hcobordism} of Smale's h-cobordism theorem \cite{smale1962structure} to continue the isolated invariant set through gradient flows to the ``simplest possible'' such set containing the fewest number of critical points.
Concatenating one of these continuations with the reversal of another yields a continuation between a given pair of isolated invariant sets.

Now suppose that the pair of isolated invariant sets are asymptotically stable equilibria $x_\vo$, $x_\vt$ for a pair of smooth vector fields $\vo$, $\vt$ on a manifold $M$.
The nicest possible continuation result would assert existence of a smooth homotopy $(H_t)_{t\in I}$ of vector fields from $\vo$ to $\vt$ and a smooth family $(x_t)_{t\in I}$ of equilibria $x_t$ for $H_t$ such that the set $$Z=\bigcup_{t\in I}\{(t,x_t)\}$$ is asymptotically stable for the vector field  $(t,x)\mapsto (0,H_t(x))$ on $I\times M$ (cf. Remark~\ref{rem:weak-vs-strong-family-asy-stab-eq} below).
Moreover, if $x_\vo=x_\vt$, then one would like $x_t=x_\vo$ for all $t\in I$.
However, Reineck's techniques do not seem to yield either conclusion. %(e.g. perturbing a gradient to make it Morse-Smale \cite[Cor.~2.2]{reineck1991continuation}).

Recently, Jongeneel proved that the latter conclusion does hold in a weak sense: if $y\in \R^n$ is globally asymptotically stable for a pair of locally Lipschitz vector fields on $\R^n$ with $n\neq 5$, then the \emph{semiflows} generated by these vector fields are homotopic through continuous semiflows (perhaps not generated by vector fields) globally asymptotically stabilizing $y$. 
The proof relies in part on the stable homeomorphism theorem (see e.g. \cite[p.~247]{edwards1984solution}).
On the other hand, our smooth techniques yield the stronger conclusion given by the following theorem, which also  solves Jongeneel's ``main open problem'' \cite[sec.~5]{jongeneel2024while} for $n\neq 4,5$ (see also \cite[sec.~5]{kvalheim2023obstructions}). 

\begin{Th}\label{th:local-path-conn}
	Let $M$ be a connected $C^\infty$ manifold of dimension $n\neq 4,5$ and $\vo$, $\vt$ be locally Lipschitz vector fields on $M$ with locally asymptotically stable equilibria $x_\vo, x_\vt \in M$.
	Then there is a locally Lipschitz homotopy $H\colon I\times M \to TM$ from $\vo$ to $\vt$ and a $C^\infty$ function $I\to M$, $t\mapsto x_t$ such that $H_t$ is a complete vector field for all $t\in (0,1)$, $x_0=x_\vo$, $x_1=x_\vt$, and $Z\coloneqq \bigcup_{t\in I}\{(t,x_t)\}$ is a locally asymptotically stable compact set of equilibria for the vector field 
   	 \begin{equation*}
   	(t,x)\mapsto (0,H_t(x))
     \end{equation*}
     on $I\times M$.
     Moreover, we can further arrange that
\begin{itemize}
	\item $x_t=x_\vo$ for all $t\in I$ if $x_\vo=x_\vt$, 
	\item $Z$ is globally asymptotically stable  if $x_\vo, x_\vt$ are,  and
	\item $H$ is $C^r$ with $r\in \N_{\geq 1}\cup \{\infty\}$ if  $\vo$, $\vt$ are $C^r$. 
\end{itemize}	
\end{Th}

\begin{Rem}\label{rem:weak-vs-strong-family-asy-stab-eq}
		 In particular, $x_t$ is locally asymptotically stable for  $H_t$ for each $t\in I$.
		 However, this property is strictly weaker than the conclusion of Theorem~\ref{th:local-path-conn}, as the example $H_t(x)=t^4x-x^3$ of vector fields on $\R$  demonstrates.
		 Here we can take $x_t=t^2$, $t\in I$.
		 Then each $x_t$ is asymptotically stable for $H_t$, but $Z=\bigcup_{t\in I}\{(t,x_t)\}$ is not asymptotically stable for the vector field $(t,x)\mapsto (0,H_t(x))$ since every neighborhood of $(0,0)\in Z$ in $I\times \R^n$ contains trajectories not converging to $Z$.
\end{Rem}

%\iffalse
\begin{proof}
Let $\bas_\vo$, $\bas_\vt$ be the basins of attraction of $x_\vo$ for $\vo$ and $x_\vt$ for $\bas_\vt$.
Let $\lyap_\vo\colon \bas_\vo\to [0,\infty)$, $\lyap_\vt\colon \bas_\vt \to [0,\infty)$ be corresponding proper $C^\infty$ Lyapunov functions (\cite[Thm~3.2]{wilson1969smooth},  \cite[Sec.~6]{fathi2019smoothing}).
Note that $\bas_\vo$, $\bas_\vt$ are both diffeomorphic to $\R^n$ \cite[Thm~3.4]{wilson1967structure}.

First suppose that $\bas_\vo = \bas_\vt = M$.
Then the proof is completed by smoothly concatenating (as in \eqref{eq:gamma-def}) the straight-line homotopy from $\vo$ to $-\nabla \lyap_\vo$ followed by the $C^\infty$ homotopy from $-\nabla \lyap_\vo$ to $-\nabla \lyap_\vt$ through globally asymptotically stable vector fields given by Corollary~\ref{co:vf-trivial-homotopy-groups} (if $x_\vo = x_\vt$) or Corollary~\ref{co:varying-eq-min-trivial-homotopy-groups} (if $x_\vo\neq x_\vt$)  followed by the straight-line homotopy from $-\nabla \lyap_\vt$ to $\vt$.

We now consider the remaining case.	
By Lemma~\ref{lem:h-cobordism}, $\lyap_\vo^{-1}([0,1])$ and $\lyap_\vt^{-1}([0,1])$ are both diffeomorphic to $D^n$.
Hence by the isotopy uniqueness of $C^\infty$ embeddings of discs \cite[Thm~III.3.6]{kosinski1993differential}, we can assume without loss of generality that $x_\vo = x_\vt\eqqcolon y$ and $\lyap_\vo^{-1}([0,1])=\lyap_\vt^{-1}([0,1])$.
Define $S= \lyap_\vo^{-1}((0,1))=\lyap_\vt^{-1}((0,1))$ and note that $S$ is diffeomorphic to $\R^n$.

Since $y$ is globally asymptotically stable for the restricted vector fields $\vo|_S$ and $\vt|_S$, we may define a locally Lipschitz (and $C^r$ if $\vo$, $\vt$ are $C^r$) homotopy $g\colon I\times S\to TS$ from $g_0 = \vt|_S$ to $g_1 = \vo|_S$ through vector fields as follows.
Similarly to above, let $g$ be the smooth concatenation of the straight-line homotopy from $\vt|_S$ to $-\nabla \lyap_\vt|_S$ followed by the $C^\infty$ homotopy through globally asymptotically stable vector fields from $-\nabla \lyap_\vt|_S$ to $-\nabla \lyap_\vo|_S$ provided by Corollary~\ref{co:vf-trivial-homotopy-groups}  followed by the straight-line homotopy from $-\nabla \lyap_\vo|_S$ to $\vo|_S$.
Hence $Z=I\times \{y\}$ is globally asymptotically stable for the vector field $(t,x)\mapsto (0,g_t(x))$ on $I\times S$ \cite[Lem.~3.1]{salamon1985connected}.
Thus, there exists a neighborhood $N\subset S$ of $y$ such that, for each $t\in I$ and $x\in N\setminus \{y\}$, the backward trajectory of $g_t$ through $x$ is not contained in $N$.

Let $\varphi \in C^\infty(M, [0,\infty))$ be $0$ on a neighborhood of  $M\setminus S$ and $1$ on $N$.
Let $\psi\in C^\infty(I,I)$ be $0$ on a neighborhood of $\{0,\frac{2}{3}\}$ and $1$ on a neighborhood of $\{\frac{1}{3},1\}$.
Define a homotopy $h\colon I\times M\to TM$ from $\vo$ to $\vt$ through vector fields by
\begin{equation*}
		h_t(x)=\begin{cases}
			(1-\psi(t))\vo(x)+\psi(t)\varphi(x)\vo(x), & 0\leq t \leq \frac{1}{3}\\
			\varphi(x)g_{\psi(t)}(x), & \frac{1}{3} \leq t \leq \frac{2}{3}\\
			(1-\psi(t))\varphi(x)\vt(x) + \psi(t)\vt(x), & \frac{2}{3}\leq t \leq 1
		\end{cases}.
\end{equation*}
Note that $h$ is locally Lipschitz and $C^r$ if $\vo$, $\vt$ are $C^r$.
Since each $h_t$ coincides with some $g_s$ on $N$, the backward trajectory of each $h_t$ through any $x\in N\setminus \{y\}$ is not contained in $N$.
Thus, $Z = I\times \{y\}$ is asymptotically stable for the vector field $(t,x)\mapsto (0,h_t(x))$ on $I\times M$ \cite[Lem.~3.1]{salamon1985connected}.

Finally, let $\eta\colon I\to [0,\infty)$ be a $C^\infty$ function satisfying $\eta^{-1}(0)=\{0,1\}$, and let $\norm{\cdot}$ be the norm of a complete Riemannian metric on $M$.
Define a homotopy $H\colon I \times M \to TM$ from $\vo$ to $\vt$ through vector fields by
\begin{equation*}
	H_t(x)=\frac{1}{1+\eta(t)\norm{h_t(x)}^2}h_t(x).
\end{equation*}
Note that $H$ is locally Lipschitz and $C^r$ if $\vo$, $\vt$ are $C^r$, and $H_t$ is complete for each $t\in (0,1)$ since $\norm{H_t}$ is bounded for each $t\in (0,1)$.
Since $H$ coincides with $h$ multiplied by a positive function, $Z$ is asymptotically stable for $(t,x)\mapsto (0,H_t(x))$.
\end{proof}
%\fi

\subsection{Smoothly trivial bundles of Lyapunov sublevel sets}\label{subsec:trivial-lyap-bundles}

The next result is the main tool used in section~\ref{subsec:param-hartman-grobman-morse}.

\begin{Th}\label{th:global-trivial-sublevel-bundles}
	Let $\param$ be a compact $C^\infty$ manifold with boundary and $\lyap\colon \param \to \fun^\infty(\R^n)$ be $C^\infty$.
	Assume that either (i) $n\neq 4,5$ and $\dim \param \leq 1$ or (ii) $n\leq 3$.
	Then for any $c > 0$, 
	\begin{equation}\label{eq:th:global-trivial-sublevel-bundles}
		\pr_1\colon \{(p,x)\in \param \times \R^n\colon \lyap(p)(x)\leq c\}\to P
	\end{equation}
	is a $C^\infty$ fiber bundle isomorphic to the trivial disc bundle $\param \times D^n\to \param$. 
\end{Th}

%\iffalse
\begin{proof}
	Define $\lyapt\colon \param \to \fun^\infty(\R^n)$ by $\lyapt(p)(x)=x^2$.
	Fix any pair of distinct points $\theta_1, \theta_2 \in \sph^1$.
	Corollary~\ref{co:varying-eq-min-trivial-homotopy-groups} applied to the manifold with boundary $\param'\coloneqq\sph^1 \times \param$ and neatly embedded submanifold $S\coloneqq\{\theta_1,\theta_2\}\times \param\subset \param'$ furnishes  a $C^\infty$ map $\lyapth\colon \sph^1\times \param\to \fun^\infty(\R^n)$ restricting to $\lyap$ on $\{\theta_1\}\times \param$ and $\lyapt$ on $\{\theta_2\}\times \param$.
    The implicit function theorem, Ehresmann's lemma, and Lemma~\ref{lem:h-cobordism} imply that  the projection 
    	\begin{equation}\label{eq:lyapunov-sublevel-bundle}
    	\{(\theta,p,x) \in \sph^1\times \param \times \R^n\colon \lyapth(\theta,p)(x)\leq c\} \to \sph^1\times \param
    \end{equation}
    onto the first two factors is a  $C^\infty$ fiber bundle with fibers diffeomorphic to $D^n$.

    Let $J\subset \sph^1$ be a closed interval with boundary $\partial J = \{\theta_1,\theta_2\}$.
    Since the restriction of the bundle \eqref{eq:lyapunov-sublevel-bundle} over $J\times \param$ in turn restricts over $\partial J \times \param$ to 
    	\begin{equation}\label{eq:two-lyapunov-sublevel-bundles}
    	\begin{tikzcd}
    		& \{\lyap(p)(x)\leq c\} \arrow{d}\\
    		& \param
    	\end{tikzcd}
    	\qquad		\bigcup
    	\begin{tikzcd}
    		& \{\lyapt(p)(x)\leq c\} \arrow{d}\\
    		& \param
    	\end{tikzcd},
    \end{equation}
    it follows that the two bundles in \eqref{eq:two-lyapunov-sublevel-bundles} are $C^\infty$ isomorphic (say, by smooth parallel transport in the total space  of the bundle \eqref{eq:lyapunov-sublevel-bundle} covering the curves $t\mapsto (t,p)$ in $J\times \param$ with respect to an Ehresmann connection). 
    Since the bundle on the right side of \eqref{eq:two-lyapunov-sublevel-bundles} is isomorphic to the trivial bundle $\param \times D^n\to \param$, we are done.
\end{proof}
%\fi

\subsection{Parametric Hartman-Grobman theorem and degenerate Morse lemma}\label{subsec:param-hartman-grobman-morse}

We first prove a parametric version of the Hartman-Grobman theorem \cite{hartman1960lemma,grobman1959homeomorphism} for families of asymptotically stable equilbria that are not necessarily hyperbolic.
This result extends a non-parametric global result \cite[Thm~2]{kvalheim2025global} sketched by Gr\"{u}ne, Sontag, and Wirth \cite[p.~133]{gruene1999twist} that extends an earlier local result of Coleman \cite{coleman1965local,coleman1966addendum}.

\begin{Th}\label{th:parametric-global-linearization}
Let $\param$ be a compact $C^\infty$ manifold with boundary and $\Phi\colon \R\times \param\times \R^n\to \param\times\R^n$ be the flow of a complete uniquely integrable continuous vector field $(p,x)\mapsto (0,\vo(p,x))$ on $\param \times \R^n$ such that $\vo(p,\cdot)$ has a globally asymptotically stable equilibrium for each $p\in \param$.
If either (i) $n\neq 4,5$ and $\dim \param \leq 1$ or (ii) $n \leq 3$,  there is a fibered homeomorphism $$h\colon \param \times \R^n\to \param \times \R^n, \qquad (p,x)\mapsto (p,h_p(x))$$ such that
\begin{equation}\label{eq:th:parametric-global-linearization}
	h(\Phi^t(p,x)) = (p,e^{-t} h_p(x)) \quad \text{for all} \quad (t,p,x)\in \R\times \param \times \R^n.
\end{equation}
Moreover, if $\Phi\in C^k$ with $k\in \N_{\geq 1}\cup \{\infty\}$, then $h$ restricts to a $C^k$ diffeomorphism on the complement of the set of equilibria. 
\end{Th}

%\iffalse
\begin{proof}		
Define $Z\subset \param \times \R^n$ to be the set of equilibria.
Since $Z$ is a compact (globally) asymptotically stable set for $\Phi$ \cite[Lem.~3.1]{salamon1985connected}, there is a corresponding proper $C^\infty$ Lyapunov function $V\colon \param \times \R^n \to [0,\infty)$ (\cite[Thm~3.2]{wilson1969smooth}, \cite[Sec.~6]{fathi2019smoothing}).

Set $L\coloneqq \{(p,x)\colon \lyap(p)(x)=1\}$.
Theorem~\ref{th:global-trivial-sublevel-bundles} furnishes a $C^\infty$ diffeomorphism $$f\colon \param \times \sph^{n-1}\to L , \quad (p,x) \mapsto (p,f_p(x)).$$

Since for each $(p,x)\not \in Z$ the trajectory $t\mapsto \Phi^t(p,x)$  converges to $Z$ and crosses $L$ exactly once and transversely, the map $$\R\times L\to (\param \times \R^n)\setminus Z$$ given by the restriction of $\Phi$ to $\R\times L$ is a homeomorphism and $C^k$ diffeomorphism if $\Phi\in C^k$, with inverse $$g=(\tau,\rho)\colon (\param \times \R^n)\setminus Z\to \R\times L$$ satisfying $\tau(p,x)\to -\infty$ as $(p,x)\to Z$ (cf. \cite[p.~327]{wilson1967structure}).
Define $h\colon \param \times \R^n\to \param \times \R^n$  by		
       \begin{equation*}
			h(p,x)=(p,h_p(x))=(p,e^{\tau(p,x)}f_p( \rho(p,x)))
		\end{equation*}
for $(p,x)\not \in Z$ and $h_p(x)=0$ for $(p,x)\in Z$.
Then $h$ is a fibered homeomorphism that, if $\Phi\in C^k$, also restricts to a $C^k$ diffeomorphism $(P\times \R^n)\setminus Z\to \param \times (\R^n\setminus \{0\})$. 
Finally, $h$ satisfies \eqref{eq:th:parametric-global-linearization} since $$\rho\circ \Phi^t|_{(\param \times \R^n)\setminus Z}=\rho$$ and $$\tau \circ \Phi^t|_{(\param \times \R^n)\setminus Z}=\tau - t.$$
\end{proof}
%\fi

The following result is a parametric version of a result of Gr\"{u}ne, Sontag, and Wirth \cite[Prop.~1]{gruene1999twist} (see also \cite[Prop.~2]{kvalheim2025global}), which can be viewed as a (global) extension of the Morse lemma \cite{morse1932calculus} %(see \cite[Lem.~2.2]{milnor1963morse})
 to minima of non-Morse functions.

\begin{Th}\label{th:parametric-global-morseification}
Let $\param$ be a compact $C^\infty$ manifold with boundary and $\lyap\colon \param \times \R^n\to [0,\infty)$ be a $C^\infty$ map such that $\lyap(p,\cdot)\in \fun^\infty(\R^n)$ for all $p\in \param$.
If either (i) $n\neq 4,5$ and $\dim \param \leq 1$ or (ii) $n \leq 3$, there is a fibered homeomorphism $$h\colon \param \times \R^n\to \param \times \R^n, \qquad (p,x)\mapsto (p,h_p(x))$$ that restricts to a $C^\infty$ diffeomorphism on the complement of the set of fiberwise critical points and satisfies 
\begin{equation}\label{eq:th:parametric-global-morseification}
\lyap(p,x)=\norm{h_p(x)}^2 \quad \text{for all} \quad (p,x)\in \param \times \R^n.
\end{equation}
\end{Th}

%\iffalse
\begin{proof}
Define $Z\subset \param \times \R^n$ to be the set of fiberwise critical points.
Let $\Phi$ be the $C^\infty$ maximal local flow of $\nabla_x \lyap/\norm{\nabla_x \lyap}^2$ on $(\param \times \R^n)\setminus Z$.
Then
\begin{equation}\label{eq:morse-comp-shift}
	\lyap \circ \Phi^t = \lyap + t,
\end{equation}
so the domain of $\Phi$ is
\begin{equation*}%\label{eq:morse-Phi-domains}
	\begin{split}
\{(t,p,x)\colon \lyap(p,x) > - t\}.
	\end{split}
\end{equation*}

Set $L\coloneqq V^{-1}(1)$.
By Theorem~\ref{th:global-trivial-sublevel-bundles}, there is a $C^\infty$ fiber-preserving diffeomorphism $f\colon L\to \param \times \sph^{n-1}$  satisfying $\pr_1\circ f = \pr_1$.

By \eqref{eq:morse-comp-shift} we may define a fibered $C^\infty$ map $\rho\colon (\param \times \R^n)\setminus Z\to \param \times \sph^{n-1}$ by
\begin{equation*}
	\rho(p,x)=(p,\rho_p(x))=\Big(p,f\circ \Phi^{1-\lyap(p,x)}(x)\Big).
\end{equation*}
Define $h\colon \param \times \R^n \to \param \times \R^n$ by
\begin{equation*}
	h(p,x)=(p,h_p(x))=\left(p,\rho_p(x)\sqrt{\lyap(p,x)}\right)
\end{equation*}
for $(p,x)\not \in Z$ and $h_p(x)=0$ for $(p,x)\in Z$.
Then $h$ is a fibered homeomorphism that satisfies \eqref{eq:th:parametric-global-morseification} by construction and restricts to a $C^\infty$ diffeomorphism $(\param\times \R^n)\setminus Z\to \param \times (\R^n\setminus \{0\})$ since $\sqrt{\cdot}$ is $C^\infty$ on $(0,\infty)$.
\end{proof}
%\fi

\subsection{Obstructions to parametric asymptotic stabilization}\label{subsec:obstructions-param-asymptotic-stab}

For motivation, first consider the vector fields $\vo \in \AN_0^\infty(\R^{2n})$, $\vt \in \stab_0^\infty(\R^{2n})$ given by $\vo(x)=x$ and $\vt(x)=-x$.
The Hopf indices $\text{ind}_\vo(0)$, $\text{ind}_\vt(0)$ of $0$ for these vector fields coincide, where for  $E\in \AN_0^0(\R^n)$ the index $\text{ind}_E(0)$ is the Brouwer degree of
\begin{equation*}
	\frac{E}{\norm{E}}|_{\sph^{n-1}} \colon \sph^{n-1}\to \sph^{n-1}.
\end{equation*}
Indeed, $$\text{ind}_F(0)=\deg(\id_{\sph^{n-1}})=1=(-1)^{2n}=\deg(x\mapsto -x)=\text{ind}_G(0).$$

Given a vector field $\vo\in \AN_0^1(\R^n)$, necessary conditions for asymptotic stability of $0$ that do not depend on more than the homotopy class of
\begin{equation*}
\vo|_{\R^n \setminus \{0\}}\colon \R^n \setminus \{0\} \to \R^n \setminus \{0\}	
\end{equation*}
have been studied in the dynamics and control literatures.
All such necessary conditions \cite{brockett1983asymptotic,krasnoselskii1984geometrical,coron1990necessary,kvalheim2022necessary} contain less or equal information to $\text{ind}_F(0)$, so  they cannot detect the instability of $0$ for $\vo(x)=x$ when $n$ is even (see \cite{jongeneel2023topological,kvalheim2023relationships} for recent discussions of the literature).

However, the results of \S \ref{sec:relative-htpy-groups} \emph{can} be used to detect instability within \emph{families} $H\colon \param \to \AN_0^1(\R^n)$ of vector fields containing $\vo(x)=x$ even when $n$ is an even number, without using information beyond the homotopy class of 
\begin{equation*}
	\param \times \R^n\setminus \{0\} \to \R^n\setminus \{0\}, \quad (p,x)\mapsto H(p)(x).
\end{equation*}

For example, consider the $C^\infty$ map $H\colon \sph^1\subset \mathbb{C}\to \AN_0^\infty(\R^2)\cong \AN_0^\infty(\mathbb{C})$ defined by $H(\theta)(z)=\theta z.$
Can we detect the presence of $\theta\in \sph^1$ such that $H(\theta) \not \in \stab_0^\infty(\mathbb{C})$ using only such homotopical information?
(Note that $\theta = 1$ is an example.)

The map 
\begin{equation}\label{eq:diag-h-map}
\sph^1 \to \sph^1, \quad \theta\mapsto \frac{H(\theta)(\theta)}{|H(\theta)(\theta)|}=\theta^2
\end{equation}
has winding number $2$.
On the other hand, if $H$ was $\stab_0^\infty(\mathbb{C})$-valued, then it would be nullhomotopic through $\AN_0^\infty(\mathbb{C})$-valued maps to the constant map $\sph^1\to \{z\mapsto -z\}$ by Theorem~\ref{th:nullhomotopic-inclusion-s-an}.
Such a nullhomotopy would induce a homotopy $I\times \sph^1\to \sph^1$ from \eqref{eq:diag-h-map} to the antipodal map $\theta\mapsto -\theta$, which has winding number $1$.
Since $1\neq 2$ and winding number is a homotopy invariant, it follows that $H$ cannot be $\stab_0^\infty(\mathbb{C})$-valued.
Thus, there indeed exists $\theta\in \sph^1$ such that $H(\theta) \not \in \stab_0^\infty(\mathbb{C})$.

\appendix

\section{Morse minima and hyperbolic equilibria}\label{sec:morse-hyp}
In this appendix we study for $r\in \N_{\geq 1}\cup \{\infty\}$ the subspaces 
\begin{equation*}
	\morsfun_0^{r+1}(\R^n)\subset \fun_0^{r+1}(\R^n), \quad \morsfun^{r+1}(\R^n)\subset \fun^{r+1}(\R^n)
\end{equation*}
consisting of Morse functions and the subspaces 
\begin{equation*}
\hypstab_0^r(\R^n) \subset \stab_0^r(\R^n), \quad \hypstab^r(\R^n) \subset \stab^r(\R^n), \quad \hypstabc_0^r(\R^n)\subset \stabc_0^r(\R^n), \quad \hypstabc^r(\R^n)\subset \stabc^r(\R^n)
\end{equation*}
consisting of vector fields with hyperbolic equilibria.
(See section~\ref{subsec:func-spaces-non-grass} for the definitions of the ambient spaces.)
The punchline is that all of these subspaces are weakly contractible for any $n$, in contrast with the main results of this paper.

A real symmetric matrix is positive definite if its eigenvalues are all positive, and a real square matrix is Hurwitz if its eigenvalues all have negative real part.
Let $\PDS(n)\subset \GL^+(n)$ and $\Hur(n)\subset \GL(n)$ denote the sets of positive definite symmetric matrices and Hurwitz matrices, which are both open subsets of $\GL(n)$.

Note that the Hessian matrix of any $\lyap \in \morsfun^2(\R^n)$ at its minimum belongs to $\PDS(n)$, and the derivative matrix of any $\vo\in \hypstab^1(\R^n)$ at its equilibrium belongs to $\Hur(n)$.
Additionally, $\PDS(n)$ and $\Hur(n)$ can be respectively identified with the subsets of quadratic forms in $\morsfun_0^{\infty}(\R^n)$ and linear vector fields in $\hypstab_0^\infty(\R^n)$.

\begin{Lem}\label{lem:matrices-contractible}
	$\PDS(n)$ and $\Hur(n)$ are both contractible via $C^\infty$ homotopies.
\end{Lem}

%\iffalse 
\begin{proof}
Since $\PDS(n)$ is an open convex cone, the straight-line homotopy from $\id_{\PDS(n)}$ to the constant map $\PDS(n)\to \{I_{n\times n}\}$ is $C^\infty$ and $\PDS(n)$-valued.

Next, let $\textnormal{spec}(B)\subset \mathbb{C}$ denote the set of eigenvalues of a matrix $B\in \R^{n\times n}$.
Since 
\begin{equation*}
\textnormal{spec}((1-t)A -t I_{n\times n}) = (1-t)\textnormal{spec}(A)-t,
\end{equation*} 
the straight-line homotopy from $\id_{\Hur(n)}$ to the constant map $\Hur(n)\to \{-I_{n\times n}\}$ is $C^\infty$ and $\Hur(n)$-valued.
\end{proof}
%\fi

We first consider the case $r=\infty$.
Lemma~\ref{lem:matrices-contractible} implies the second statement below.

\begin{Th}\label{th:morse-hyp-contractible}
There is a strong deformation retraction of $\morsfun^\infty(\R^n)$ onto $\PDS(n)$ preserving $\morsfun_0^\infty(\R^n)$, and there is a strong deformation retraction of $\hypstab^\infty(\R^n)$ onto $\Hur(n)$ preserving $\hypstabc^\infty(\R^n)$, $\hypstab_0^\infty(\R^n)$, and $\hypstabc_0^\infty(\R^n)$.
Thus $\morsfun_0^{\infty}(\R^n)$, $\hypstab_0^\infty(\R^n)$, $\hypstabc_0^\infty(\R^n)$, $\morsfun^{\infty}(\R^n)$, $\hypstab^\infty(\R^n)$, and $\hypstabc^\infty(\R^n)$ are all contractible. 
\end{Th}

\begin{proof}
First note that, by translation, there are strong deformation retractions of $\morsfun^\infty(\R^n)$ onto $\morsfun_0^\infty(\R^n)$ and $\hypstab^\infty(\R^n)$ onto $\hypstab_0^\infty(\R^n)$, with the latter preserving $\hypstabc^\infty(\R^n)$ (see the proof of Lemma~\ref{lem:varying-minima-equilibria}).
Thus, it suffices to show that $\PDS(n)$ is a strong deformation retract of $\morsfun_0^{\infty}(\R^n)$ and that there is a strong deformation retraction of $\hypstab_0^\infty(\R^n)$ onto $\Hur(n)$ that preserves $\hypstabc_0^\infty(\R^n)$.

Define $H\colon I\times \morsfun_0^\infty(\R^n)\to \morsfun_0^\infty(\R^n)$ by
\begin{equation*}%\label{eq:beta-def-for-funnel-case}
	H_t(\lyap)(x) = 
	\begin{cases}
		\frac{1}{(1-t)^2}\lyap((1-t)x), & 0 \leq t < 1\\
		\ip{x}{\frac{1}{2}D_0^2 \lyap x}, & t=1\\
	\end{cases}.
\end{equation*}
The chain rule and fundamental theorem of calculus imply that
\begin{align*}
	\lyap(x) &= \ip{x}{\int_0^1 s \int_0^1 D_{usx}^2 \lyap\, du \, ds \cdot x}\\
	&=  \ip{x}{\eta(\lyap)(x)x}
\end{align*}
for a map $\eta\colon \morsfun_0^\infty(\R^n) \to C^\infty(\R^n, \R^{n\times n})$ that is continuous with respect to the compact-open $C^\infty$ topology (see section~\ref{subsubsection:generalities}) and satisfies $\eta(\lyap)(0)=\frac{1}{2}D_0^2\lyap$, so
\begin{equation*}
	H_t(\lyap)(x) = \ip{x}{\eta(\lyap)((1-t)x)x}
\end{equation*}
for all $t\in I$, $\lyap\in \morsfun_0^\infty(\R^n)$, and $x\in \R^n$.
Moreover, $H_t(\lyap)=\lyap$ for all $\lyap \in \PDS(n)$ and $t\in I$, so $H$ is a strong deformation retraction of $\morsfun_0^\infty(\R^n)$ to $\PDS(n)$.

 Next, define $J\colon I\times \hypstab_0^\infty(\R^n) \to \hypstab_0^\infty(\R^n)$ by 
\begin{equation*}
	J_t(\vo)(x) = \begin{cases}
		\frac{1}{(1-t)}\vo((1-t)x), & 0 \leq t < 1\\
		D_0\vo x, & t=1
	\end{cases}.
\end{equation*}
The chain rule and fundamental theorem of calculus imply that 
\begin{align*}
	\vo(x) &= \int_0^1 D_{sx}\vo\, ds \cdot x\\
	&= \delta(\vo)(x)x
\end{align*}
for a map $\delta \colon \hypstab_0^\infty(\R^n) \to C^\infty(\R^n, \R^{n\times n})$ that is continuous with respect to the compact-open $C^\infty$ topology and satisfies $\delta(\vo)(0)=D_0 \vo$, so
\begin{equation*}
	J_t(\vo)(x) = \delta(\vo)((1-t)x)x
\end{equation*}
for all $t\in I$, $\vo\in \hypstab_0^\infty(\R^n)$, and $x\in \R^n$.
Moreover, $J_t(\vo)=\vo$ for all $\vo\in \Hur(n)$ and $t\in I$, so $J$ is a strong deformation retraction of $\hypstab_0^\infty(\R^n)$ onto $\Hur(n)$.
And $J_t$ clearly preserves $\hypstabc_0^\infty(\R^n)$ for all $t\in I$.
\end{proof}

For a refined smoothness statement and the case $r < \infty$, we have the following weak contractibility result.
The proof is similar to that of the preceding theorem, but more technical.
(See section~\ref{subsection:maps-into-function-spaces} for the definition of ``continuous and $C^s$''.)

\begin{Th}\label{th:morse-hyp-weakly-contractible}
	Fix $r\in \N_{\geq 1}\cup \{\infty\}$ and let $\F$ denote any of $\morsfun_0^{r+1}(\R^n)$, $\hypstab_0^r(\R^n)$, $\hypstabc_0^r(\R^n)$, $\morsfun^{r+1}(\R^n)$, $\hypstab^r(\R^n)$, or $\hypstabc^r(\R^n)$. 
	Then $\F$ is weakly contractible.
	In fact, if $\param$ is a $C^\infty$ manifold with boundary and $f\colon \param \to \F$ is a continuous and $C^s$ map, then there is a continuous and $C^s$ homotopy $I\times \param \to \F$ from $f$ to a constant map.
\end{Th}	

\begin{proof}
We first refine the notation from the theorem statement.	
Let $\F$ now denote any of $\morsfun^{r+1}(\R^n)$, $\hypstab^r(\R^n)$, or $\hypstabc^r(\R^n)$, and let $\F_0\subset \F$ denote the subspaces of functions that are $0$ at $0$.
Let	 $\F^\infty\subset \F$ and $\F_0^\infty\subset \F_0$ denote those functions that are additionally $C^\infty$.
%Equip $\F_0^\infty$ and $\F^\infty$ with the compact-open $C^\infty$ topology.

Using the implicit function theorem of Hildebrandt and Graves \cite[Thm~4]{hildebrandt1927implicit} and the standard method of convolution with a mollifier \cite[sec.~2.2--2.3]{hirsch1976differential}, we approximate $f$ in the strong $C^s$ and fiberwise strong $C^r$ topologies (Remark~\ref{rem:fiberwise-compact-open-strong-topologies}) by a $C^\infty$ map $\tilde{f}\colon \param \to \F^\infty$ (section~\ref{subsection:maps-into-function-spaces}) such that the straight-line homotopy from $f$ to $\tilde{f}$ is an $\F$-valued map $\tilde{\alpha}\colon I\times \param \to \F$ and 
\begin{align*}
	Z &\coloneqq \{(t,p,x)\colon \tilde{\alpha}_t(p)(x)=0\}\\
	&= \{(t,p,z_t(p))\}
\end{align*} 
for a $C^s$ map $z\colon I\times \param \to \R^n$ such that $z_1\colon \param \to \R^n$ is $C^\infty$ \cite[Thm~10.3, Lem.~20.1]{abraham1967transversal}.
Define a continuous and $C^s$ map $\alpha\colon I\times \param \to \F$ by 
\begin{equation*}
	\alpha_t(p)(x)=\tilde{\alpha}_t(p)(x+\varphi(t)z_t(p)),
\end{equation*}
where we take $\varphi \equiv 1$ if $f$ is $\F_0$-valued, and otherwise take $\varphi\in C^\infty(I,\R)$ to be $0$ at $0$ and $1$ at $1$.
Then $\alpha$ is $\F_0$-valued in the case that $f$ is, and in either case $\alpha$ is a continuous and $C^s$ homotopy $I\times \param \to \F$ from $\alpha_0=f$ to a $C^\infty$ map $g\colon \param \to \F_0^\infty$.
Note also that $\alpha$ is $\hypstabc^r(\R^n)$-valued if $f$ is.

Next, if $\F=\morsfun^{r+1}(\R^n)$, then define $\beta\colon I\times \param \to \morsfun_0^\infty(\R^n)$ by
\begin{equation*}%\label{eq:beta-def-for-funnel-case}
	\beta_t(p)(x) = 
	\begin{cases}
		\frac{1}{(1-t)^2}g(p)((1-t)x), & 0 \leq t < 1\\		
		\ip{x}{\frac{1}{2}D_0^2 g(p)x}, & t=1
	\end{cases}.
\end{equation*}
The chain rule and fundamental theorem of calculus imply that 
\begin{align*}
	g(p)(x) &= \ip{x}{\int_0^1 s \int_0^1 D_{usx}^2g(p)\, du \, ds \cdot x}\\
	&=  \ip{x}{\eta(p)(x)x}
\end{align*}
for a $C^\infty$ map $\eta\colon \param\to C^\infty(\R^n, \R^{n\times n})$ satisfying $\eta(p)(0)=\frac{1}{2}D_0^2 g(p)$, so
\begin{equation*}
	\beta_t(p)(x) = \ip{x}{\eta(p)((1-t)x)x}
\end{equation*}
for all $t\in I$, $p\in \param$, and $x\in \R^n$.
Thus, $\beta$ is a $C^\infty$ homotopy $I\times \param\to \morsfun_0^\infty(\R^n)$ from $\beta_0 = g$ to a $C^\infty$ map $h\colon \param \to \PDS(n)\subset \morsfun_0^\infty(\R^n)$.

If instead $\F$ is  $\hypstab^r(\R^n)$ or $\hypstabc^r(\R^n)$, then define $\beta\colon I\times \param \to \hypstab_0^\infty(\R^n)$ by
\begin{equation*}
	\beta_t(p)(x) = \begin{cases}
		\frac{1}{(1-t)}g(p)((1-t)x), & 0 \leq t < 1\\
		D_0g(p)x, & t=1		
	\end{cases}.
\end{equation*}
The chain rule and fundamental theorem of calculus imply that 
\begin{align*}
	g(p)(x) &= \int_0^1 D_{sx}g(p)\, ds \cdot x\\
	&= \delta(p)(x)x
\end{align*}
for a $C^\infty$ map $\delta \colon \param \to C^\infty(\R^n, \R^{n\times n})$ satisfying $\delta(p)(0)=D_0 g(p)$, so
\begin{equation*}
	\beta_t(p)(x) = \delta(p)((1-t)x)x
\end{equation*}
for all $t\in I$, $p\in \param$, and $x\in \R^n$.
Thus, $\beta$ is a $C^\infty$ homotopy $I\times \param \to \hypstab_0^\infty(\R^n)$ from $\beta_0 = g$ to a $C^\infty$ map $h\colon \param \to \Hur(n)\subset \hypstab_0^\infty(\R^n)$.
Moreover, clearly $\beta$ is $\hypstabc_0^\infty(\R^n)$-valued if $g$ is, which is the case if $\F = \hypstabc^r(\R^n)$.

Smoothly concatenating (as in \eqref{eq:gamma-def}) the homotopies $\alpha$ followed by $\beta$ followed by the $C^\infty$ nullhomotopy $\gamma$ of $h$ provided by Lemma~\ref{lem:matrices-contractible}  yields the desired continuous and $C^s$ nullhomotopy of $f$.
\end{proof}

Unlike the similar Corollaries~\ref{co:funnels-trivial-homotopy-groups}, \ref{co:vf-trivial-homotopy-groups}, \ref{co:complete-vf-trivial-homotopy-groups}, \ref{co:varying-eq-min-trivial-homotopy-groups}, the following corollary makes no restrictions on $n$ or $\dim \param$.

\begin{Co}\label{co:morse-hyp-bvp}
Fix $r\in \N_{\geq 1}\cup \{\infty\}$ and let $\F$ denote any of $\morsfun_0^{r+1}(\R^n)$, $\hypstab_0^r(\R^n)$, $\hypstabc_0^r(\R^n)$, $\morsfun^{r+1}(\R^n)$, $\hypstab^r(\R^n)$, or $\hypstabc^r(\R^n)$. 
If $\param$ is a $C^\infty$ manifold with boundary, $S\subset \param$ is a closed subset and $C^\infty$ neatly embedded submanifold of either $\param$ or $\partial \param$, and $f_S\colon S \to \F$ is continuous and $C^s$, then there is a continuous and $C^s$ map $f\colon \param \to \F$ extending $f_S$.
\end{Co}

\begin{Rem}\label{rem:morse-hyp-bvp}
	Taking $S = \partial \param$ gives an existence result for a boundary value problem.
\end{Rem}

\begin{proof}
Using Theorem~\ref{th:morse-hyp-weakly-contractible} and a closed tubular neighborhood $U\subset \param$ of $S$, we may extend $f_S$ to a continuous and $C^s$ map $\tilde{f}\colon U\to \F$ that is constant on a neighborhood of $\partial U$.
Extending $\tilde{f}$ to be constant on $\param \setminus U$ yields the desired continuous and $C^s$ map $f\colon \param \to \F$ restricting to $f_S$ on $S$.
\end{proof}

\begin{Rem}\label{rem:mors-hyp-applications}
For Morse functions and vector fields with hyperbolic equilibria, using Corollary~\ref{co:morse-hyp-bvp} instead of Corollaries~\ref{co:vf-trivial-homotopy-groups}, \ref{co:varying-eq-min-trivial-homotopy-groups} in the proofs of section~\ref{subsec:param-hartman-grobman-morse} yields stronger versions of the results in that section, without any restrictions on $n$ or $\dim \param$.
\end{Rem}

	\bibliographystyle{plain}
	\bibliography{ref}
	
	%\clearpage    
	%\input{author-response.tex}    
\end{document}